\documentclass[a4paper,11pt, twoside]{article}

\usepackage{a4}
\usepackage{geometry} 
\usepackage{amsmath}
\usepackage{amssymb}
\usepackage{amsthm}
\usepackage{graphicx}
\usepackage{caption,subcaption}
\usepackage{multirow}
\usepackage{xstring}
\usepackage[utf8]{inputenc}
\usepackage{hyperref}
\usepackage{amstext,amsbsy,amsopn,eucal,enumerate}
\usepackage{tikz}
\usepackage{pgfplots}
\usepackage{tabularx}

\usepackage{soul}
\usepackage{algorithm} %ctan.org\pkg\algorithms
\usepackage{algpseudocode}

\usepackage[natbibapa,nodoi]{apacite}% Citation support using apacite.sty. Commands using natbib.sty MUST be deactivated first!
\def\R{\mathbb{R}}
\def\C{\mathbb{C}}

\newcommand{\expn}{\operatorname{e}}

\newcommand{\diag}{\operatorname{diag}}

\newcommand{\beq}{\begin{equation}}
\newcommand{\eeq}{\end{equation}}
\newcommand {\mat}      [1] {\left[\begin{array}{#1}}
\newcommand {\rix}          {\end{array}\right]}
\newcommand {\smat}      [1] {\left[\begin{smallmatrix}{#1}}
\newcommand {\srix}          {\end{smallmatrix}\right]}
\newcommand {\s}      [1] {\begin{smallmatrix}{#1}}
\newcommand {\se}          {\end{smallmatrix}}
\newcommand{\trace}{\operatorname{tr}}

\newtheorem{defn}{Definition}[section]
\newtheorem{remark}{Remark}

\newtheorem{lem}[defn]{Lemma}
\newtheorem{prop}[defn]{Proposition} %Proposition entspricht Satz

\newtheorem{thm}[defn]{Theorem}

\def\addlegendimage{\csname pgfplots@addlegendimage\endcsname}

\newlength\fheight
\newlength\fwidth

  \newcommand{\matlab}{MATLAB\textsuperscript{\textregistered}}
  \newcommand{\intel}{Intel\textsuperscript{\textregistered}}

\definecolor{mycolor1}{rgb}{0.00000,0.44700,0.74100}%
\definecolor{mycolor2}{rgb}{0.85000,0.32500,0.09800}%
\definecolor{mycolor3}{rgb}{0.92900,0.69400,0.12500}%
\definecolor{mycolor4}{rgb}{0.49400,0.18400,0.55600}%
\definecolor{mycolor5}{rgb}{0.46600,0.67400,0.18800}%
\title{Model order reduction for bilinear systems with non-zero initial states -- different approaches with error bounds}
\author{Martin Redmann\thanks{Martin Luther University Halle-Wittenberg, Institute of Mathematics, Theodor-Lieser-Str. 5, 06120 Halle (Saale), Germany, Email: {\tt 
martin.redmann@mathematik.uni-halle.de}.}\and Igor Pontes Duff \thanks{Max Planck Institute for Dynamics of Complex Technical Systems, Magdeburg, Germany, Email: {\tt 
pontes@mpi-magdeburg.mpg.de}.}
}

\begin{document}

\maketitle

\begin{abstract}
In this paper, we consider model order reduction for bilinear systems with non-zero initial conditions. We discuss choices of Gramians for both the homogeneous and the inhomogeneous parts of the system individually and prove how these Gramians characterize the respective dominant subspaces of each of the two subsystems. Proposing different, not necessarily structure preserving, reduced order methods for each subsystem, we establish several strategies to reduce the dimension of the full system. For all these approaches, error bounds are shown depending on the truncated Hankel singular values of the subsystems. Besides the error analysis, stability is discussed. In particular, a focus is on a new criterion for the homogeneous subsystem guaranteeing the existence of the associated Gramians and an asymptotically stable realization of the system.
\end{abstract}


%%%%%%%%%%%%%%%%%%%%%%

\section{Introduction}

In this paper, we study model order reduction (MOR) techniques for the following system with non-zero initial states:
\begin{subequations}\label{fullsys}
\begin{align}\label{biode}
\dot x(t) &= Ax(t) +  Bu(t) + \sum_{k=1}^m N_k x(t) u_k(t),\quad x(0) = x_0 = X_0 v_0,\\ \label{biout}
y(t) &= Cx(t),\quad t\geq 0,
  \end{align}
  \end{subequations}
  where $A, N_k \in \R^{n\times n}$, $B \in \R^{n\times m}$, and $C \in \R^{p\times n}$. Moreover, $x$ is the state vector, $y$ the quantity of interest and the columns of $X_0\in \R^{n\times q}$ span all initial states $x_0$ that are considered here, i.e., there exist $v_0\in\R^q$ such that $x_0 = X_0 v_0$. We assume that the matrix $A$ is Hurwitz, meaning 
that $\sigma(A) \subset \C_{-}=\{z\in\mathbb C:\quad \Re(z)<0\}$, where $\sigma(\cdot)$ denotes the spectrum of a matrix and $\Re(\cdot)$ represents the real part of a complex number. Furthermore, let $u = (\begin{matrix} u_1, u_2, \dots, u_m\end{matrix})^\top\in L^2$, i.e., \begin{align*}
                   \left\|u\right\|_{L^2}^2:= \int_0^\infty \left\|u(s)\right\|^2_2 ds = \int_0^\infty u^\top(s) u(s) ds <\infty.
                  \end{align*}
There exist many different MOR techniques for bilinear systems when $x_0=0$, e.g., methods that are balancing related \citep{typeIBT, hartmann, redmannspa2, redstochbil},  optimization/interpolation based \citep{breiten_benner, flagggug} and data-driven \citep{loewener_bil}. However, many applications involve non-zero initial states such that a study for MOR for \eqref{fullsys} is essential. Several approaches in this context have been established for linear systems \citep{BauerBenner_inhom, inhom_lin, spa_inhom, HeinReisAn, Matze_inhom}. There is no straightforward generalization of these techniques to \eqref{fullsys}, since the study of transfer functions and fundamental solutions is much more involved for bilinear systems. In this work, we choose an approach that relies on estimates for fundamental solutions of bilinear systems that originate in \citep{h2_bil}. These estimates enable a detailed theoretical analysis for an ansatz that conceptionally extends the one used in \citep{inhom_lin}.                                                                                                                                                                                                                                                                                                                                                                                                                                                                                                                                                                                                                                                                                                                                                                                                                                                                                                                                                                                                                                                                                                                                                                                                                                                                                                                                                                                                                                                                                                                                                                                                                                                                                                                                                                                                                                                                                                                                                                                                                                                                                                                                                                                                                                                                                                                                                                                                                                                                                                                                                                                                                                                                                                                                                                                                                                                                                                                                                                      
The general idea is to split \eqref{fullsys} into two subsystems. System 
\begin{subequations}\label{bilinear_x0}
\begin{align}\label{bilinear_x0_state}
 \dot x_{x_0}(t) &= Ax_{x_0}(t) +  \sum_{k=1}^m N_k x_{x_0}(t) u_k(t),\quad x_{x_0}(0) = X_0 v_0,\\ \label{bilinear_x0_out} y_{x_0}(t) &= Cx_{x_0}(t).
\end{align}
\end{subequations}
involves the initial condition and
\begin{subequations}\label{bilinear_B}
\begin{align}\label{bilinear_B_state}
  \dot x_B(t) &=  Ax_B(t) +  Bu(t) + \sum_{k=1}^m N_k x_B(t) u_k(t),\quad x_B(0) = 0,\\ \label{bilinear_B_out}
  y_B(t) &= C x_B(t)
\end{align}
\end{subequations}
captures the inhomogeneous part of \eqref{fullsys}. Consequently, we have $x=x_{x_0} + x_B$ and $y = y_{x_0} +  y_B$. The above splitting was considered in \citep{huang2021splitting}, where the authors discuss a balancing approach to produce reduced order models (ROMs). Additionally,
MOR of \eqref{fullsys} was proposed in \citep{inhom_bil} based on a different splitting, that was originally described in \citep{d1974realization}. However, theoretical questions remain open for these approaches such as the error analysis. Our main goal in this work is to propose new MOR schemes for the bilinear systems \eqref{fullsys}  with general initial conditions possessing computable error bounds based on the neglected singular values of the system. To this aim, we propose and investigate balanced truncation (BT) and the singular perturbation approximation (SPA) method for subsystem \eqref{bilinear_x0}. As a consequence, we are able to show $L^2$-error bounds for those MOR schemes. For reasons of consistency, we use the MOR results available in the literature (in particular from \citep{redmannstochbilspa, redstochbil}) to construct reduced models for subsystem \eqref{bilinear_B} with $L^2$-error bounds.  \smallskip

Notice that the need for MOR of bilinear systems with non-zero initial states is higher than for linear systems since there is an essential difference between both cases. For linear systems, it is required that several initial states are of interest in order to motivate applying MOR to the homogeneous equation. However, the homogeneous bilinear system \eqref{bilinear_x0} is control dependent such that MOR can already pay off for a single initial condition ($X_0=x_0$ and $v_0=1$) if system evaluations for multiple controls are desired. The individual reduction of \eqref{bilinear_x0} and \eqref{bilinear_B} has several advantages. As for linear systems, one subsystem can have a higher reduction potential than the other. Hence, reduced order dimensions can be chosen differently, but the actual benefit of the splitting goes beyond this degree of freedom. In addition, it turns out that using different Gramians and different structures of the reduced systems can be beneficial. 
In particular, this splitting method combined with the proposed MOR schemes is shown to preserve stability and to possesses computable error bounds depending on truncated singular values, with the price  of having to reduce two different subsystems instead of potentially just one. However,  is worth noticing that up to our knowledge there is no other system theoretical MOR method for  bilinear systems with non-zero initial conditions that does not involve any type of splitting. \smallskip

In this work, we discuss several Gramian based approaches in which subsystems \eqref{bilinear_x0} and \eqref{bilinear_B} are reduced separately. This leads to reduced order models
\begin{align}\label{rom1}
\dot {\tilde x}_{x_0}(t) = \tilde A_{x_0} {\tilde x}_{x_0}(t) +  \sum_{k=1}^m \tilde N_{x_0, k} \tilde x_{x_0}(t) u_k(t),\quad \tilde x_{x_0}(0) = \tilde X_0 v_0,\quad \tilde y_{x_0}(t) = \tilde C_{x_0} \tilde x_{x_0}(t),\end{align}
($\tilde A_{x_0}, \tilde N_{x_0, k} \in \mathbb R^{r_{x_0}\times r_{x_0}}$, $\tilde X_0 \in \R^{r_{x_0}\times q}$ and $\tilde C_{x_0} \in \R^{p\times r_{x_0}}$) approximating \eqref{bilinear_x0} and to reduced systems
\begin{align}\label{rom2}
\begin{aligned}
             \dot {\tilde x}_B(t) &= \tilde A_B {\tilde x}_B(t)+\tilde B u(t)+\sum_{k=1}^m \left(\tilde N_{B, k} {\tilde x}_B(t) + \tilde E_{k} u(t)\right)u_k(t),\quad\tilde x_B(0) = 0, \\
    \tilde y_B(t) &=\tilde C_B {\tilde x}_B(t)+ \tilde D u(t),
            \end{aligned}
            \end{align}
($\tilde A_{B}, \tilde N_{B, k} \in \mathbb R^{r_{B}\times r_{B}}$, $\tilde B \in \R^{r_{B}\times m}$ and $\tilde C_{B} \in \R^{p\times r_{B}}$, $\tilde D\in\R^{p\times m}$ and $\tilde E_{k}\in\R^{r_{B}\times m}$) approximating \eqref{bilinear_B} 
with $\tilde x_{x_0}(t)\in \R^{r_{x_0}}$ and $\tilde x_{B}(t)\in \R^{r_{B}}$, where $r_{x_0}, r_B\ll n$ and all above matrices are of suitable dimension. The goal is to choose \eqref{rom1} and \eqref{rom2} such that $y\approx \tilde y_{x_0} + \tilde y_B$. Notice that we will introduce both a structure preserving variant ($\tilde D, \tilde E_{k} =0$) of \eqref{rom2} and a scheme with $\tilde D, \tilde E_{k} \neq 0$ (see Section \ref{MOR_B}). For the second method the additional quadratic control term in \eqref{rom2} is vital. Else, the approximation would be worse and the associated error bound in Theorem 
\ref{special_error_bound2} could not be established.
\smallskip

In this paper, we provide estimates that explain how the considered Gramians characterize dominant subspaces in both \eqref{bilinear_x0} and \eqref{bilinear_B}. Such a result for \eqref{bilinear_x0} has not even been established in the linear case. These estimates give a motivation for different Gramian based MOR techniques proposed in this paper 
without directly using control concepts such as reachability or observability. Moreover, we prove error bounds for all methods studied within this paper, closing a gap in the analysis of such schemes. However, the main focus is on analyzing \eqref{rom1}, since different results on properties of \eqref{rom2} already exist in the literature. 

The paper is organized as follows. In Section  \ref{section2}, we recall some basic results on bilinear systems. Therein, the state evolution is characterized using the fundamental solution of a bilinear system. Additionally, we state some intermediate results that will later be required to prove MOR error bounds. In Section \ref{section3}, we proposed Gramians for the subsystems \eqref{bilinear_x0} and \eqref{bilinear_B}.  Additionally, we show the eigenspaces of those Gramians are associated with dominant subspaces of the subsystem dynamics. Based on these Gramians, in Section \ref{sec:balancing}, we propose two different model reduction schemes, namely BT and SPA, for both  \eqref{bilinear_x0} and \eqref{bilinear_B}. Those methodologies were tailored such that the reduced models obtained possess an error bound with respect to the $L^2$-norm. Indeed, in Subsection \ref{MOR_x0}, we demonstrate that the MOR schemes for \eqref{bilinear_x0}  possesses an $L^2$-error bound depending on the neglected singular values. In Subsection \ref{MOR_B}, we recall variants of the BT and SPA from the literature leading to $L^2$-error bounds for \eqref{bilinear_B}. It is important to emphasize that the choice of Gramians and MOR schemes were made such that the overall procedure is stability preserving, and an $L^2$-error bound depending on the truncated singular values can be established. Finally, in Section \ref{sec:NumRes}, numerical experiments are conducted in order to illustrate the efficiency of the algorithms.     

\section{Solution representation and fundamental solutions}\label{section2}

We begin with such essential concepts and estimates in this section that are required to investigate the error and the stability for the particular MOR schemes proposed later.\smallskip

The fundamental solution $\Phi$ to \eqref{biode} will play a very important role in the analysis of the MOR techniques investigated in this paper. $\Phi$ represents a basis for the solution to the homogeneous state equation ($B=0$). Its precise definition is as follows:
\begin{defn}\label{defn_fund}
Given that $s\leq t$, the fundamental solution to \eqref{biode} is a matrix-valued function $\Phi$ satisfying
 \begin{align*}
 \Phi(t,s) = I +\int_s^t A \Phi(v,s) dv + \sum_{k=1}^m \int_s^t N_k \Phi(v,s) u_k(v) dv,
\end{align*}
where $I$ is the identity matrix. If $s=0$, we set $\Phi(t):=\Phi(t, 0)$.
\end{defn}
This fundamental solution can now be used to derive an explicit representation for the state variable. It is beneficial since it generally holds in contrast to Volterra series expansions of the state that require more restrictive assumptions in order to ensure convergence. Moreover, working with the fundamental solution is vital in the context of MOR since our  error estimates rely on the representation given in the following lemma.
\begin{lem}\label{lem_sol_rep}
The solution to \eqref{biode} for $0\leq t_0\leq t$ is given by
  \begin{align*}
 x(t) = \Phi(t, t_0) x(t_0) + \int_{t_0}^t \Phi(t,s) B u(s) ds.
\end{align*}
\end{lem}
\begin{proof}
Using that $\Phi(t,s) = \Phi(t) \Phi^{-1}(s)$, the result follows by applying the product rule to $\Phi(t) g(t)$, where $g(t) := \Phi^{-1}(t_0)x(t_0) + \int_{t_0}^t \Phi^{-1}(s) B u(s) ds$. 
\end{proof}
 Unfortunately, the fundamental solution of a bilinear system is control dependent and hence we have $\Phi(t, s) \neq \Phi(t-s)$ (no semigroup property of $\Phi(\cdot)$) making it infeasible to directly apply Lemma \ref{lem_sol_rep} in order to obtain error bounds. Therefore, an estimate on $\Phi$ is needed in order to extract the dependence on $u$. This result is formulated in the lemma below. It is, e.g., the essential ingredient to derive the stability criterion in Theorem \ref{cond_stab} and to prove the preliminary error bound of Lemma \ref{basis_bound}. In the following, we write $M_1\leq M_2$ if $M_2-M_1$ is symmetric positive semidefinite given that $M_1$ and $M_2$ are symmetric matrices.
\begin{lem}\label{fund_est}
Let $\Phi$ be the fundamental solution according to Definition \ref{defn_fund}, $K\geq 0$ and $\gamma>0$. Then,
 \begin{align*}
\Phi(t, s) K \Phi^\top(t, s) \leq  \exp\left\{\int_s^t \left\|\gamma u^{0}(v)\right\|_2^2 dv\right\}  Z_\gamma(t-s),
\end{align*}
where $ Z_\gamma(t)$, $t\geq 0$, satisfies the matrix differential equation
\begin{align}\label{eqZ}
 \dot { Z}_\gamma(t) = A  Z_\gamma(t) +  Z_\gamma(t) A^\top +\frac{1}{\gamma^2}\sum_{k=1}^m N_k Z_\gamma(t) N_k^\top ,\quad Z_\gamma(0) = K,
\end{align}
and $u^0$ is the vector of control functions entering the bilinear part\begin{align}\label{equ0}
u^{0}=(u_1^{0}\ u_2^{0}\, \ldots\, u_m^{0})^\top\quad \text{with}\quad u_k^{0} \equiv \begin{cases}
  0,  & \text{if }N_k = 0\\
  u_k, & \text{else}.
\end{cases}
\end{align}
\end{lem}
\begin{proof}
The proof is stated in Appendix \ref{appendix_section2}.
\end{proof}
Lemma \ref{fund_est} is a variation of the results from Lemmas 2.2, 2.3 and 2.4 in \citep{h2_bil}. The constant $\gamma$ in Lemma \ref{fund_est} is essential to achieve asymptotic stability of \eqref{eqZ}. Based on this stability, Gramians for \eqref{bilinear_x0} will be introduced in Section \ref{sec:Gram_x0}. However, a bit less than asymptotic stability is needed, as the following theorem shows. It contains a sufficient condition for the existence of Gramians that we need to establish error bounds. This criterion is related to a matrix inequality and can be seen as an extended notion of stability for \eqref{eqZ}. We will also see later that ROMs \eqref{rom1} based on balancing generally satisfy such a condition.
\begin{thm}\label{cond_stab}
Let $\gamma>0$ and $Z_\gamma(\cdot, X_0 X_0^\top)$ the solution to \eqref{eqZ} with $K= X_0 X_0^\top$. If there exists a matrix $X>0$ such that \begin{align}\label{extended_stab}
     A  X+ X  A^\top + \frac{1}{\gamma^2}\sum_{k=1}^m N_{k} X N_{k}^\top \leq -X_0 X_0^\top.                                     
                                          \end{align}
  Then, \eqref{eqZ} is stable meaning that \begin{align}\label{bounded_Z} \sigma\left(I\otimes A + A\otimes I+ \frac{1}{\gamma^2}\sum_{k=1}^{m} N_k \otimes N_k\right)\subset \overline{\mathbb C_-}
                \end{align}
with $\overline{\mathbb C_-}$ denoting the closure of ${\mathbb C_-}$. Moreover, there is a constant $c>0$ such that $\left\|Z_\gamma(t, X_0 X_0^\top)\right\|_2 \lesssim \expn^{-c t}$\;\footnote{The relation $\lesssim$ tells that the left side can be bounded by the right side up to an unspecified constant.}, i.e., the initial condition $K= X_0 X_0^\top$ yields exponential decay. In particular, we can construct a matrix $V \in \R^{n \times \tilde n}$, $\tilde n\leq n$, with $V^\top V = I$ providing a projected system with coefficients $\tilde A = V^\top A V$,   $\tilde X_0 = V^\top X_0$ and $\tilde N_{k} =  V^\top{N}_{k}V$. This reduced  system with fundamental solution $\tilde \Phi$ has an asymptotically stable equation \eqref{eqZ}, i.e., it holds that
\begin{align}\label{stab_min_rel} \sigma\left(I\otimes \tilde A + \tilde A\otimes I+ \frac{1}{\gamma^2}\sum_{k=1}^{m} \tilde N_k \otimes \tilde N_k\right)\subset {\mathbb C_-}
                \end{align}
and it has no reduction error in the sense that
\[ \Phi(t)X_0 = V \tilde{\Phi}(t)\tilde X_0.\]
\end{thm}
\begin{proof}
The proof is given in Appendix \ref{appendix_section2}.
\end{proof}
Theorem \ref{cond_stab} shows that if \eqref{extended_stab} is satisfied, the bilinear system represented by the matrices $A, N_k$ with initial conditions encoded by the matrix $X_0$ can be always reduced to a asymptotically stable system  in the sense of \eqref{stab_min_rel} with no reduction error. Below, we briefly discuss that asymptotic stability is a stronger concept than the criterion for the existence of the system Gramians given in \eqref{extended_stab}.
\begin{remark}
If \eqref{eqZ} is asymptotically stable there exists an $X>0$ such that \begin{align*}
 A  X+ X  A^\top + \frac{1}{\gamma^2}\sum_{k=1}^m N_{k} X N_{k}^\top= Y                                                                                                                           \end{align*}
given $Y<0$, see \citep{damm}. Setting $Y= -I - X_0 X_0^\top$ now  implies \eqref{extended_stab}.
\end{remark}

\section{Gramians and dominant subspaces}\label{section3}

\subsection{Gramians and dominant subspaces for \eqref{bilinear_x0}}\label{sec:Gram_x0}
We begin with investigating the homogeneous part of \eqref{biode} with non-zero initial states. 
To do so, we study two Gramians for \eqref{bilinear_x0} that provide information concerning the dominant subspaces of \eqref{bilinear_x0_state} and \eqref{bilinear_x0_out}, respectively.\smallskip

In order to identify the unimportant directions in \eqref{bilinear_x0_state} a Gramian $P_0$ is introduced below. Let $Z_\gamma = Z_\gamma(t, X_0 X_0^\top)$ as in \eqref{eqZ} and $K= X_0 X_0^\top$. The existence of the Gramians requires the asymptotic stability of \eqref{eqZ} which is stronger than $\sigma(A)\subset \mathbb C_-$. However, we can enforce this stronger type of stability by a sufficiently large $\gamma>0$ providing \begin{align}\label{stab}
 \sigma( A\otimes I+I\otimes  A+ \frac{1}{\gamma^2}\sum_{k=1}^m  N_k \otimes  N_k)\subset \mathbb C_-.  
\end{align}
The rescaled matrices $\frac{1}{\gamma} N_k$ in \eqref{stab} are associated to
the following equivalent reformulation of \eqref{bilinear_x0_state}:
\begin{align*}
 \dot x_{x_0}(t) = Ax_{x_0}(t) +  \sum_{k=1}^m \frac{1}{\gamma} N_k x_{x_0}(t) \gamma u_k(t),\quad x_{x_0}(0) = X_0 v_0,
\end{align*}
but it goes along with an enlarged control energy in the bilinearity. Now, we define \begin{align*}
  P_{0}:=\int_0^\infty Z_\gamma(s, X_0 X_0^\top) ds.                                                                                                                                                                                                                                                                                           \end{align*}
The dependence of $P_0$ on $\gamma$ is not explicitly indicated to simplify the notation. By definition of $P_0$ and the asymptotic stability of \eqref{eqZ}, we can immediately see that $P_0$ solves \begin{align}\label{reach_gram_bil_x0}
 A P_{0} + P_{0} A^\top + \frac{1}{\gamma^2}\sum_{k=1}^m N_k P_{0} N_k^\top = -X_0 X_0^\top.
\end{align}
We are now ready to establish an estimate identifying redundant information in \eqref{bilinear_x0_state}. Therefore, let us introduce an orthonormal basis $(p_{0, i})$ of eigenvectors of $P_0$. Consequently, we can write $x_{x_0}(t) = \sum_{i=1}^n \langle x_{x_0}(t), p_{0, i}\rangle_2\, p_{0, i}$. The following estimate for $\langle x_{x_0}(t), p_{0, i}\rangle_2$ allows us to find directions $p_{0, i}$ which barely contribute to the dynamics.
\begin{prop}\label{unimportant_state0}
Let $x_{x_0}$ denote the solution to \eqref{bilinear_x0_state} and $\gamma>0$ such that \eqref{stab} holds. Then, 
 \begin{align}\label{reach_est_bil}
  \left\|\langle x_{x_0}(\cdot), p_{0, i}\rangle_2 \right\|_{L^2} \leq \lambda_{0, i}^{\frac{1}{2}} \exp\left\{0.5\left\|\gamma u^{0}\right\|_{L^2}^2\right\} \left\|v_0\right\|_2,          
                    \end{align}
where $\lambda_{0, i}$ is the eigenvalue associated to $p_{0, i}$.
\end{prop}
\begin{proof}
Based on Lemma \ref{lem_sol_rep} we find that 
  \begin{align*}
 x_{x_0}(t) = \Phi(t) x_0 = \Phi(t) X_0 v_0.
\end{align*}
Exploiting this leads to \begin{align*}
  \int_0^t \langle x_{x_0}(s), p_{0, i}\rangle_2^2 ds &=  \int_0^t \langle \Phi(s) X_0 v_0, p_{0, i}\rangle_2^2 ds = \int_0^t \langle  v_0, X_0^\top \Phi^\top(s) p_{0, i}\rangle_2^2 ds  \\
  &\leq \left\|v_0\right\|^2_2 p_{0, i}^\top  \int_0^t \Phi(s) X_0 X_0^\top \Phi^\top(s)  ds\; p_{0, i}
  \end{align*}
  using the inequality of Cauchy-Schwarz. Using Lemma \ref{fund_est}, we obtain\begin{align*}
  \int_0^t \langle x_{x_0}(s), p_{0, i}\rangle^2 ds
   &\leq \left\|v_0\right\|^2_2 \exp\left\{\int_0^t \left\|\gamma u^{0}(v)\right\|_2^2 dv\right\}   p_{0, i}^\top  \int_0^t Z_\gamma(s, X_0 X_0^\top)  ds\; p_{0, i} \\
   &\leq \left\|v_0\right\|^2_2\exp\left\{\left\|\gamma u^{0}\right\|_{L^2}^2\right\} p_{0, i}^\top  P_0\; p_{0, i}
   =\left\|v_0\right\|^2_2\exp\left\{\left\|\gamma u^{0}\right\|_{L^2}^2\right\}  \lambda_{0, i}.
                        \end{align*}
\end{proof}
Consequently, $x_{x_0}$ is small in the direction of an eigenvector $p_{0, i} = p_{0, i}(\gamma)$ of $P_0$ associated to a small eigenvalue $\lambda_{0, i} = \lambda_{0, i}(\gamma)$. This means that eigenspaces corresponding to small eigenvalues of $P_0$ are less relevant and hence can be neglected. \smallskip 

Let us now turn our attention to the choice of Gramians and the related dominant subspaces of \eqref{bilinear_x0_out}. 
We introduce the matrix-valued function $Z_\gamma^*= Z^*_\gamma(t, C^\top C)$ satisfying \begin{align}\label{eqZad}
 \dot { Z}_\gamma^*(t) = A^\top Z_\gamma^*(t) +  Z_\gamma^*(t) A+\frac{1}{\gamma^2}\sum_{k=1}^m N_k^\top Z_\gamma^*(t) N_k,\quad Z_\gamma^*(0) = C^\top C,
\end{align}
where the superscript $*$ indicates that the Lyapunov operator defining the right side of \eqref{eqZad} is the adjoint operator of the one entering \eqref{eqZ}. Let us further assume that \eqref{stab} holds. Then, we define \begin{align}\label{defnitionQ}                                                                                                                                                                                                                                
Q:=\int_0^\infty Z_\gamma^*(s, C^\top C) ds.                                                                                                                                                                                                \end{align}
By definition of $Q$ and the asymptotic stability of \eqref{eqZad}, we have \begin{align}\label{obs_gram_bil_x0}
 A^\top Q + Q A + \frac{1}{\gamma^2}\sum_{k=1}^m N_k^\top Q N_k = -C^\top C.
\end{align}
Let $0\leq t_0 <\infty$. We now expand $x_{x_0}(t_0)$ using 
an orthonormal basis $(q_{i})$ of eigenvectors of $Q$, i.e., we write $x_{x_0}(t_0) = \sum_{i=1}^n \langle x_{x_0}(t_0), q_{i}\rangle_2 q_{i}$.  The goal is to identify the directions $q_{i}$ which do not contribute significantly to the output $y_{x_0}$ on the interval $(t_0, \infty)$. We exploit the representation in Lemma \ref{lem_sol_rep} and obtain for $t\geq t_0$ that \begin{align}\label{outputexpansion}                                                                                                                                                                                                                                                                                                                                                                                            
  y_{x_0}(t) = C \Phi(t, t_0) x_{x_0}(t_0) =  \sum_{i=1}^n  C \Phi(t, t_0) q_{i} \langle x_{x_0}(t_0), q_{i}\rangle_2.                                                                                                                                                                                                                            \end{align}
Eigenvectors $q_{i}$ can now be neglected if the respective summand in \eqref{outputexpansion} is small in some norm. These summands are now analyzed in the following theorem.
\begin{prop}\label{redundant_out}
Let $(q_{i})$ be an orthonormal basis of eigenvectors of the Gramian $Q$ and $\gamma>0$ such that \eqref{stab} holds. Then, 
 \begin{align}\label{obs_est_bil}
  \left(\int_{t_0}^\infty \left\|C \Phi(t, t_0) q_{i}\right\|^2_2 dt\right)^{\frac{1}{2}} \leq \mu_{i}^{\frac{1}{2}} \exp\left\{0.5\left\|\gamma u^{0}\right\|_{L^2}^2\right\},          
                    \end{align}
where $\mu_{i}$ is the eigenvalue associated to $q_{i}$.
\end{prop}
\begin{proof}
With Lemma \ref{fund_est}, we find \begin{align*}
  \int_{t_0}^\infty \left\|C \Phi(t, t_0) q_{i}\right\|^2_2 dt &=     \int_{t_0}^\infty q_{i}^\top  \Phi(t, t_0) C^\top C \Phi(t, t_0) q_{i} dt\\
  &= \int_{t_0}^\infty \trace\left(C \Phi(t, t_0) q_{i}q_{i}^\top\Phi^\top(t, t_0) C^\top\right) dt \\&\leq 
  \int_{t_0}^\infty \trace\left(C \exp\left\{\int_{t_0}^t \left\|\gamma u^{0}(v)\right\|_2^2 dv\right\}  Z_\gamma(t-t_0, q_{i}q_{i}^\top) C^\top\right) dt \\&\leq 
  \exp\left\{\left\|\gamma u^{0}\right\|_{L^2}^2\right\}\int_{0}^\infty \trace\left(C  Z_\gamma(s, q_{i}q_{i}^\top) C^\top\right) ds\\
 &= \exp\left\{\left\|\gamma u^{0}\right\|_{L^2}^2\right\}\trace\left(C^\top C  \int_{0}^\infty Z_\gamma(s, q_{i}q_{i}^\top)ds \right).
                   \end{align*}
$\int_{0}^\infty Z_\gamma(s, q_{i}q_{i}^\top)ds$ solves \eqref{reach_gram_bil_x0} with right hand side $q_{i}q_{i}^\top$. Inserting \eqref{obs_gram_bil_x0} for $C^\top C $ above, we can see that $\trace\left(C^\top C  \int_{0}^\infty Z_\gamma(s, q_{i}q_{i}^\top)ds \right)= q_{i}^\top Q q_{i} = \mu_i$. This concludes the proof.
\end{proof}
Estimate \eqref{obs_est_bil} now tells us that $q_i= q_i(\gamma)$ is an unimportant direction in $x_{x_0}(t_0)$ for each $t_0\geq 0$ if $\mu_i= \mu_i(\gamma)$ is small since these vectors have a low impact on the output $y_{x_0}(t)$, $t\geq t_0$. Consequently, eigenspaces of $Q$ corresponding to small eigenvalues can be removed from the system. 

\subsection{Gramians and dominant subspaces for \eqref{bilinear_B}}\label{sec:Gram_B}

We introduce a reachability Gramian $P_B$ as a positive definite solution to 
\begin{align}\label{newgram2}
 A^\top P_B^{-1}+P_B^{-1}A+ \frac{1}{\gamma^2}\sum_{k=1}^m  N^\top_k P_B^{-1} N_k \leq -P_B^{-1}BB^\top P_B^{-1}.
                                       \end{align}
Such a solution exists given that \eqref{stab} holds, see \citep[Lemma III.1]{dammbennernewansatz} or more generally  \citep[Proposition 3.1]{redmannspa2}                                     
Notice that an inequality is considered in \eqref{newgram2}, since the existence of a positive definite solution of the associated equality is not ensured. $P_B$ identifies directions in the state equation \eqref{bilinear_B_state} that can be removed from the system. To see this, let $(p_{B, i})$ an orthonormal basis of eigenvectors of $P_B$, such that $x_{B}(t) = \sum_{i=1}^n \langle x_{B}(t), p_{B, i}\rangle_2\; p_{B, i}$. As in Proposition \ref{unimportant_state0} an estimate for $\langle x_{B}(t), p_{B, i}\rangle_2$ can be found. However, the norm is a different one.
\begin{prop}
Let $x_{B}$ denote the solution to \eqref{bilinear_B_state} and $\gamma>0$ such that \eqref{stab} holds. Then, 
 \begin{align}\label{diffreachjaneintype2}
\sup_{t\geq 0}\left\vert\langle x_B(t), p_{B, i}  \rangle_2\right\vert \leq \lambda_{B, i}^{\frac{1}{2}} \left\|u\right\|_{L^2} \exp\left(0.5 \left\|\gamma  u^0\right\|_{L^2}^2\right),
\end{align}
where $\lambda_{B, i}$ is the eigenvalue associated to $p_{B, i}$.
\end{prop}
\begin{proof}
The result for $\gamma = 1$ is a special case of \citep[Section 2.1]{redstochbil}. Rescaling $N_k x_{B}(t) u_k(t)\mapsto  \frac{1}{\gamma} N_k x_{B}(t) \gamma u_k(t)$ in \eqref{bilinear_B_state} immediately provides the desired estimate for general $\gamma$. 
\end{proof}
By \eqref{diffreachjaneintype2}, we can see that  $p_{B, i} = p_{B, i}(\gamma)$ is less relevant for the dynamics if $\lambda_{B, i} = \lambda_{B, i}(\gamma)$ is small. For that reason, one is interested in computing a $P_B$ with possibly small eigenvalues since such a solution to \eqref{newgram2} characterizes the negligible information best. Therefore, determining $P_B$ becomes an optimization problem of, e.g., 
 minimizing $\trace(P_B)$ subject to \eqref{newgram2}. In Section \ref{mor_section}, a MOR procedure is discussed that is based on $P_B$ and $Q$ (the relevance of $Q$ for \eqref{bilinear_B_out} is presented below). It turns out that the error of such MOR schemes are characterized by the truncated eigenvalues of $P_B Q$, see Theorem \ref{special_error_bound2}, meaning that a large set of small eigenvalues of $P_B Q$ yields a low approximation error. For that reason, one may also think of computing $P_B$ based on minimizing $\trace(P_B Q)$. However, numerical experiments indicate that including $Q$ in the optimization procedure leads to slightly worse results.\smallskip

The dominant subspace of \eqref{bilinear_B_out} can be found with the same Gramian $Q$, defined in \eqref{defnitionQ}                                                                                                                                                                                                                                
as in the case of $y_{x_0}$. We expand 
$x_{B}(t_0) = \sum_{i=1}^n \langle x_{B}(t_0), q_{i}\rangle_2 q_{i}$ for $0\leq t_0<\infty$. By Lemma \ref{lem_sol_rep} we have  \begin{align}\label{irrelevant_Bu}                                                                                                                                                                                                                                                                                                                                                                                           
  y_{B}(t) &= C \Phi(t, t_0) x_{B}(t_0) + \int_{t_0}^t C \Phi(t,s) B u(s) ds\\ \nonumber
  &=  \sum_{i=1}^n  C \Phi(t, t_0) q_{i} \langle x_{B}(t_0), q_{i}\rangle_2 + \int_{t_0}^t C \Phi(t,s) B u(s) ds                                                                                                                                                                                                                            \end{align}
for $t\geq t_0$. Therefore, the direction $q_i$ is less relevant if $C \Phi(t, t_0) q_{i}$ is small. The corresponding estimate for this expression has already been established in Proposition \ref{redundant_out}. Consequently, $q_{i}$ is also negligible for $y_B$ if the eigenvalue $\mu_{i}$ is small. 
That the same $Q$ is used in both subsystems is not surprising since the structural difference of \eqref{bilinear_x0} and \eqref{bilinear_B} lies in $Bu$, a term that does not depend on the state itself. Hence, it is irrelevant in the observability context looking at the second summand in \eqref{irrelevant_Bu}. One might choose a different parameter $\gamma$ in \eqref{obs_gram_bil_x0} in each subsystem. However, we believe that no large benefit can be expected by using different scalings. Therefore, we choose $\gamma$ to be the same for both systems below.

\section{Gramian-based model order reduction}\label{mor_section}

\subsection{Balancing of subsystems  \eqref{bilinear_x0} and \eqref{bilinear_B}}\label{sec:balancing}
We have seen in Sections \ref{sec:Gram_x0} and \ref{sec:Gram_B} that the eigenspaces corresponding to small eigenvalues of $P_0$ and $Q$ are not important for subsystem \eqref{bilinear_x0} and the ones of $P_B$ and $Q$ are less relevant for subsystem \eqref{bilinear_B}. Therefore, we construct a state space transformation ensuring that $P_0$ and $Q$ are diagonal and equal, meaning that $p_{0, i} = q_i = e_i$, where $e_i$ is the $i$th unit vector in $\R^n$. The $i$th diagonal entry of the diagonalized Gramians then determines how much the $i$th component of the state variable contributes to the dynamics. This procedure of simultaneously diagonalizing the Gramians is called balancing. After conducting this procedure for \eqref{bilinear_x0}, another balancing transformation is constructed for \eqref{bilinear_B}, guaranteeing that $P_B$ and $Q$ are diagonal and equal as well. Subsequently, the unimportant information in both subsystems can be removed, leading to the reduced models \eqref{rom1} and \eqref{rom2}.\smallskip

The procedure sketched above now works as follows. Based on the assumption that  $P_{0}, Q >0$, we can construct the following regular matrices and their inverses 
\begin{align}\label{balancedtrans}
\mathcal S= \Theta^{-\frac{1}{2}}\mathcal U^\top L^\top,\quad \mathcal S^{-1}=\mathcal K \mathcal V \Theta^{-\frac{1}{2}}
\quad\text{and}\quad S=\Sigma^{-\frac{1}{2}}U^\top L^\top,\quad S^{-1}=K V\Sigma^{-\frac{1}{2}},
\end{align}
where $\Theta=\diag(\theta_{1},\ldots,\theta_{n})>0$ and 
$\Sigma=\diag(\sigma_{1},\ldots,\sigma_{n})>0$  with $\theta_i$ and $\sigma_i$ being the square root of the $i$th eigenvalue of $P_{0} Q$ and $P_{B} Q$, respectively. These diagonal entries of $\Theta$ and $\Sigma$ are called Hankel singular values (HSVs) of \eqref{bilinear_x0} and \eqref{bilinear_B}. The other ingredients in \eqref{balancedtrans} are computed by the factorizations $P_{0} = \mathcal K \mathcal K^\top$, $P_{B} = K K^\top$, $Q=LL^\top$ and the singular value decompositions of $\mathcal K^\top L = \mathcal  V\Theta \mathcal U^\top$ and $K^\top L = V\Sigma U^\top$. \smallskip

Replacing $(A, X_0, C, N_k)$ by the transformed matrices \begin{align}\label{partmatricesx0}
\mathcal S{A}\mathcal S^{-1}= \smat{\mathcal A}_{11}&{\mathcal A}_{12}\\ 
{\mathcal A}_{21}&{\mathcal A}_{22}\srix,\quad \mathcal S{X_0} = \smat{X}_{0, 1}\\ {X}_{0, 2}\srix,\quad  
{C\mathcal S^{-1}} = \smat{\mathcal C}_1 &{\mathcal C}_2\srix,\quad \mathcal S{N_k}\mathcal S^{-1}= \smat {\mathcal N}_{k, 11}&{\mathcal N}_{k, 12}\\ 
{\mathcal N}_{k, 21}&{\mathcal N}_{k, 22}\srix, \end{align}
in \eqref{bilinear_x0} with ${\mathcal A}_{11},\mathcal N_{k,11}\in\R^{r_{x_0}\times r_{x_0}}$, $X_{0,1} \in \R^{r_{x_0} \times q}$, and $\mathcal{C}_1 \in \R^{p \times r_{x_0}}$, we obtain the following system 
\begin{equation}\label{sys:original_trans_x0}
\begin{aligned}
\smat {\dot {\mathbf x}}_1(t) \\ {\dot {\mathbf x}}_2(t)\srix &= \smat{\mathcal A}_{11}&{\mathcal A}_{12}\\ 
{\mathcal A}_{21}&{\mathcal A}_{22}\srix \smat {\mathbf x}_1(t) \\{\mathbf x}_2(t)\srix + \sum_{k=1}^m \smat {\mathcal N}_{k, 11}&{\mathcal N}_{k, 12}\\ 
{\mathcal N}_{k, 21}&{\mathcal N}_{k, 22}\srix \smat {\mathbf x}_1(t) \\{\mathbf x}_2(t)\srix u_k(t),\quad \smat {\mathbf x}_1(0) \\{\mathbf x}_2(0)\srix=\smat{X}_{0, 1}\\ {X}_{0, 2}\srix v_0 \\
y_{x_0}(t) &= \smat{\mathcal C}_1 &{\mathcal C}_2\srix \smat {\mathbf x}_1(t) \\{\mathbf x}_2(t)\srix,\quad t\geq 0,
\end{aligned}
\end{equation}
having the same output as \eqref{bilinear_x0}. Above, we set $\mathcal S x_{x_0}(t)=\smat {\mathbf x}_1(t) \\{\mathbf x}_2(t)\srix$. The Gramian of \eqref{sys:original_trans_x0} are \begin{align}\label{balanced_cal_S}
\mathcal S P_{0}  \mathcal S^\top = \mathcal S^{-\top}Q\mathcal S^{-1}= \Theta  = \smat{\Theta}_{1}& \\ 
 &{\Theta}_{2}\srix
\end{align}
with ${\Theta}_{1}\in\R^{r_{x_0}\times r_{x_0}}$ and ${\Theta}_{2} = \diag(\theta_{r_{x_0}+1},\ldots,\theta_{n})$ contains the $n-r_{x_0}$ smallest HSVs of the subsystem. 
\smallskip

The same way, $(A, B, C, N_k)$ is replaced by \begin{align}\label{partmatricesB}
S{A}S^{-1}= \smat{A}_{11}&{A}_{12}\\ 
{A}_{21}&{A}_{22}\srix,\quad  S{B} = \smat{B}_{1}\\ {B}_{2}\srix,\quad  
{C S^{-1}} = \smat{C}_1 &{C}_2\srix,\quad S{N_k} S^{-1}= \smat { N}_{k, 11}&{ N}_{k, 12}\\ 
{ N}_{k, 21}&{ N}_{k, 22}\srix, \end{align}
in \eqref{bilinear_B} with ${A}_{11}, N_{k,11}\in\R^{r_{B}\times r_B}$, $B_1 \in \R^{r_{B} \times m}$, and $C_1 \in \R^{p \times r_{B}}$ such that we have
\begin{equation}\label{sys:original_trans_B}
\begin{aligned}
\smat {\dot x}_1(t) \\ {\dot x}_2(t)\srix &= \smat{A}_{11}&{A}_{12}\\ 
{A}_{21}&{A}_{22}\srix \smat x_1(t) \\x_2(t)\srix +  \smat{B}_1\\ {B}_2\srix u(t) + \sum_{k=1}^m \smat {N}_{k, 11}&{N}_{k, 12}\\ 
{N}_{k, 21}&{N}_{k, 22}\srix \smat x_1(t) \\x_2(t)\srix u_k(t),\\
y_B(t) &= \smat{C}_1 &{C}_2\srix \smat x_1(t) \\x_2(t)\srix,\quad t\geq 0,
\end{aligned}
\end{equation}
where $ S x_{B}(t)=\smat {x}_1(t) \\{x}_2(t)\srix$ and the new Gramians are \begin{align*}
S P_{B}  S^\top = S^{-\top}QS^{-1}= \Sigma  = \smat{\Sigma}_{1}& \\ 
 &{\Sigma}_{2}\srix
\end{align*}
with ${\Sigma}_{1}\in\R^{r_{B}\times r_{B}}$ and ${\Sigma}_{2} = \diag(\sigma_{r_B+1},\ldots,\sigma_{n})$.
\begin{remark}
 It is important to point out that the balancing transformations $\mathcal S$ and $S$ depend on $\gamma$ since the Gramians are functions of this parameter. Consequently, the balancing realizations in \eqref{partmatricesx0} and \eqref{partmatricesB}, as well as the later ROMs, depend on $\gamma$.
\end{remark}
% \igor{
% \begin{remark} Notice that such a balancing transformation can also exist if the Gramians  are semidefinite. To illustrate this fact, let us consider a bilinear subsystem as \eqref{bilinear_x0} in given by the matrices 
% \[A = \begin{bmatrix}
% -1 & 0 & 0 
% \\
% 0 & -1 & 0
% \\
% 0 & 0 & -1 
% \end{bmatrix}, \quad N = \begin{bmatrix}
% 0 & 0 & 0 
% \\
% 1 & 0 & 0
% \\
% 0 & 0 & 0 
% \end{bmatrix}, \quad X_0 = \begin{bmatrix}
% 2 
% \\
% 0
% \\
% 0
% \end{bmatrix} \quad \text{and}~\, C = \begin{bmatrix}
% \sqrt{3} & 0 & 0 
% \\
% 0 & \sqrt{2} & 0
% \end{bmatrix}. \]
% Hence, for this systems we must have
% \[ P_{0} = Q = \begin{bmatrix}
% 2 & 0 & 0 
% \\
% 0 & 1 & 0 
% \\
% 0 & 0 & 0 
% \end{bmatrix}. \]
% As a consequence, the above subsystem is balanced even if the matrix $P_0$ and $Q$ are semidefinite.
% \end{remark}
% }
\subsection{Model order reduction for subsystem \eqref{bilinear_x0}}\label{MOR_x0}

In this section, we discuss two different MOR techniques for \eqref{bilinear_x0} that rely on the balancing procedure described in Section \ref{sec:balancing}. We already know that the state variables $\mathbf x_2$ in the balanced realization \eqref{sys:original_trans_x0} are less relevant since they are associated to the small HSVs $\theta_{r_{x_0}+1},\ldots,\theta_{n}$. A ROM \eqref{rom1} can now be obtained by neglecting $\mathbf x_2$. A first option is to truncate the second line of the state equation in \eqref{sys:original_trans_x0} and to set $\mathbf x_2(t) = 0$ in the remaining parts of the subsystem. This methods is called balanced truncation and leads to a ROM with \begin{align}\label{bt_x0}
                                                                                                                                                                                                                                                                \left(\tilde A_{x_0}, \tilde X_0, \tilde C_{x_0},  \tilde N_{x_0, k}\right) = \left(\mathcal A_{11}, X_{0, 1}, \mathcal C_{1},  \mathcal N_{k, 11}\right).                                                                                                                                                                                                                                                \end{align}
Alternatively, one can argue that due to \eqref{reach_est_bil}, $\mathbf x_2$ is close to its equilibrium (especially if the system is uncontrolled). Hence, it is in a quasi steady state, motivating to set $\dot{\mathbf x}_2(t)=0$ in \eqref{sys:original_trans_x0}. If we further neglect ${\mathcal N}_{k, 21}$ and ${\mathcal N}_{k, 22}$ in the resulting algebraic constrain in order to avoid a control dependence of the matrices in the ROM, we obtain $\mathbf x_2(t)=-\mathcal A_{22}^{-1} \mathcal A_{21} \mathbf x_1(t)$. Inserting this for $\mathbf x_2$ in \eqref{sys:original_trans_x0} leads to a ROM with  \begin{align}   \label{spa_x0}                                                                                                                                                                                                                                                                                                                                                                                                                                                            
                                                                                                                                                                                                                                                              \left(\tilde A_{x_0}, \tilde X_0, \tilde C_{x_0},  \tilde N_{x_0, k}\right) = \left(\bar {\mathcal A}, X_{0, 1}, \bar {\mathcal C},  \bar {\mathcal N}_{k}\right),                                                                                                                                                                                                                               \end{align}                                                                                                                                                                                                                                        where $\bar {\mathcal A} := \mathcal A_{11} - \mathcal A_{12} \mathcal A_{22}^{-1} \mathcal A_{21}$, $\bar {\mathcal C}:= \mathcal C_1-\mathcal C_2 \mathcal A_{22}^{-1} \mathcal A_{21}$ and $\bar {\mathcal N}_{k} := \mathcal N_{k, 11} - \mathcal N_{k, 12} \mathcal A_{22}^{-1} \mathcal A_{21}$. It is important to point out that both ROMs \eqref{bt_x0} and \eqref{spa_x0} share the same initial condition matrix $\tilde X_0$.
Notice that the structure preservation in the ROM is also desired here, which is motivated by the existence of an error bound that we prove later. This bound can only be achieved between systems having the same structure. We refer to a related SPA MOR scheme for \eqref{bilinear_B} in \citep{hartmann}, where such a reduced system was derived by an averaging principle representing a more detailed motivation than given here.
\begin{remark}
A result on stability preservation for BT has already been established in \citep{redbendamm}. Given $\Theta>0$ and $\sigma(\Theta_1)\cap \sigma(\Theta_2) = \emptyset$, it was shown that \begin{align*}
 \sigma( \mathcal A_{11}\otimes I+I\otimes  \mathcal A_{11}+ \frac{1}{\gamma^2}\sum_{k=1}^m  \mathcal N_{k, 11} \otimes  \mathcal N_{k, 11})\subset \mathbb C_-.  \end{align*}
 Whether SPA guarantees this type of stability under the same assumption is an open question. However, for SPA it can be proved that the eigenvalue of the above Kronecker matrix involving the matrices in \eqref{spa_x0} are in $\overline {\mathbb C_-}$, see \citep{redSPA}. Since $\Theta>0$ and $\sigma(\Theta_1)\cap \sigma(\Theta_2) = \emptyset$ might not be always given, stability preservation and the existence of Gramians for the two different balancing related methods are discussed in the following, only assuming $\Theta_1>0$.
\end{remark}
\begin{thm}\label{red_gram_ex}
Let $\mathcal S$ be the balanced transformation providing \eqref{balanced_cal_S} with $\Theta_1>0$ and consider the associated balanced realization
in \eqref{partmatricesx0}. Given the matrix differential equations \begin{align*}
\dot { \tilde Z}_\gamma(t) &= \tilde A_{x_0} \tilde Z_\gamma(t) +  \tilde Z_\gamma(t) \tilde A_{x_0}^\top+\frac{1}{\gamma^2}\sum_{k=1}^m \tilde N_{x_0, k} \tilde Z_\gamma(t) \tilde N_{x_0, k}^\top,\quad \tilde Z_\gamma(0) = \tilde X_0 \tilde X_0^\top,\\
 \dot { \tilde Z}_\gamma^*(t) &= \tilde A_{x_0}^\top \tilde Z_\gamma^*(t) +  \tilde Z_\gamma^*(t) \tilde A_{x_0}+\frac{1}{\gamma^2}\sum_{k=1}^m \tilde N_{x_0, k}^\top \tilde Z_\gamma^*(t) \tilde N_{x_0, k},\quad \tilde Z_\gamma^*(0) = \tilde C_{x_0}^\top \tilde C_{x_0},
\end{align*}
 the Gramians 
$\tilde P:=\int_0^\infty \tilde Z_\gamma(s) ds$ and $\tilde Q:=\int_0^\infty \tilde Z_\gamma^*(s) ds$ 
exist for reduced system \eqref{rom1} with coefficients as in \eqref{bt_x0} (BT). If instead the ROM by SPA defined in \eqref{spa_x0} is considered, the existence of $\tilde Q$ is ensured.
\end{thm}
\begin{proof}
The proof can be found in Appendix \ref{appendix_section4}.
\end{proof}      
Due to Theorem \ref{red_gram_ex} technical assumptions like $\theta_{r_{x_0}}\neq \theta_{r_{x_0}+1}$ can be omitted in the error analysis if BT is considered since the reduced Gramians will still exist. Furthermore, given $\Theta_1>0$, Theorem \ref{cond_stab}  and \eqref{romgrambt} tell us that the ROM by BT can always be reduced to a system satisfying \eqref{stab_min_rel} without causing an additional error. 
Whether $\tilde P$ generally exists for SPA remains open and is therefore always assumed below. Now, we establish error bounds for the  BT and SPA procedures. Firstly, we prove an intermediate lemma in order to show this result.
\begin{lem}\label{basis_bound}
Let $y_{x_0}$ be the output of \eqref{bilinear_x0}
and $\tilde y_{x_0}$ be the reduced order output of system \eqref{rom1}. Then, we have \begin{align*}
\left\|y_{x_0}(t) - \tilde y_{x_0}(t) \right\|_2^2 \leq \trace\left(\smat C & -\tilde C_{x_0}\srix Z_\gamma^e(t) \smat C^\top \\ -\tilde C_{x_0}^\top\srix\right)  \exp\left\{\int_0^t \left\|\gamma u^{0}(v)\right\|_2^2 dv\right\}\left\|v_0\right\|_2^2,
\end{align*}
where $ Z_\gamma^e(t)$, $t\geq 0$, satisfies the matrix differential equation
\begin{align*}
 \dot { Z}_\gamma^e(t) &= \smat A& 0\\0 &\tilde A_{x_0}\srix   Z_\gamma^e(t) +  Z_\gamma^e(t) \smat A^\top& 0\\0 &\tilde A_{x_0}^\top\srix  +\frac{1}{\gamma^2}\sum_{k=1}^m \smat N_k& 0\\0 &\tilde N_{x_0, k}\srix Z_\gamma^e(t) \smat N_k^\top& 0\\0 &\tilde N_{x_0, k}^\top\srix ,\\ Z_\gamma^e(0) &= \smat X_0\\ \tilde X_0\srix\smat X_0^\top & \tilde X_0^\top\srix.
\end{align*}
\end{lem}
\begin{proof}
 We refer to appendix \ref{appendix_section4} for the proof.
\end{proof}

\begin{thm}\label{special_error_bound}
Let $y_{x_0}$ be the output of \eqref{bilinear_x0}
given that \eqref{stab} holds for a sufficiently large $\gamma>0$. Suppose that $\tilde y_{x_0}$ is the reduced order output of system \eqref{rom1}, where the matrices $\left(\tilde A_{x_0}, \tilde X_0, \tilde C_{x_0},  \tilde N_{x_0, k}\right)$ are either given by BT stated in \eqref{bt_x0} or by SPA defined in \eqref{spa_x0} given a balancing transformation $\mathcal S$ as in \eqref{balanced_cal_S} with $\Theta_1>0$.  We assume that the reduced system Gramian $\tilde P$ for SPA exists. Defining 
 $Y =  \smat{Y}_{1} \\ Y_2 \srix$ as the unique solution to 
 \begin{align}\label{yequation}
\smat{\mathcal A}_{11}&{\mathcal A}_{12}\\ 
{\mathcal A}_{21}&{\mathcal A}_{22}\srix \smat{Y}_{1} \\ Y_2 \srix+ \smat{Y}_{1} \\ Y_2 \srix \tilde A^\top_{x_0} +\frac{1}{\gamma^2}\sum_{k=1}^{m}  \smat {\mathcal N}_{k, 11}&{\mathcal N}_{k, 12}\\ 
{\mathcal N}_{k, 21}&{\mathcal N}_{k, 22}\srix \smat{Y}_{1} \\ Y_2 \srix {\tilde N}^\top_{x_0} = -\smat{X}_{0, 1}\\ {X}_{0, 2}\srix  X_{0, 1}^\top,
    \end{align}     
using the balanced realization
in \eqref{partmatricesx0}, it holds that \begin{align}\label{interprete_bound}
 \left\|y_{x_0} - \tilde y_{x_0} \right\|_{L^2} \leq   \left(\trace(\Theta_2 \mathcal W_{x_0})\right)^{\frac{1}{2}}  \exp\left\{0.5\left\|\gamma u^{0}\right\|_{L^2}^2\right\} \left\|v_0\right\|_2.
\end{align}
The above weight is \begin{align*}
 \mathcal W_{x_0} = X_{0, 2} X_{0, 2}^\top + 2 Y_2{\mathbf A}_{21}^\top + 2\smat{\mathcal A}_{21}&{\mathcal A}_{22}\srix Y {\bar{\mathbf A}}_{21}^\top+ \frac{1}{\gamma^2}\sum_{k=1}^{m} 2 \smat
{\mathcal N}_{k, 21}&{\mathcal N}_{k, 22}\srix Y{\mathbf N}_{k, 21}^\top -    \mathbf N_{k, 21}\tilde P \mathbf N_{k, 21}^\top,
\end{align*}
where $(\mathbf A_{21}, \bar {\mathbf A}_{21}, {\mathbf N}_{k, 21}) =(\mathcal A_{21}, 0, \mathcal N_{k, 21})$ for BT and $(\mathbf A_{21}, \bar {\mathbf A}_{21}, {\mathbf N}_{k, 21})=(0, -  {\mathcal A}_{22}^{-1}{\mathcal A}_{21}, {\mathcal N}_{k, 21}-{\mathcal N}_{k, 22}{\mathcal A}_{22}^{-1}{\mathcal A}_{21})$ for SPA.
\end{thm}
\begin{proof}
We moved the proof of this theorem to Appendix \ref{appendix_section4} in order to improve the readability of the paper.
\end{proof} 
The result of Theorem \ref{special_error_bound} shows an error bound that depends on the truncated HSVs. Choosing $r_{x_0}$ such that $\Theta_2$ is small therefore ensures a small error and hence a good approximation.

\subsection{Model order reduction for subsystem \eqref{bilinear_B}}\label{MOR_B}

In this section, BT and SPA are studied for \eqref{bilinear_B}.  As for the methods considered in Section \ref{MOR_x0} they rely on the balancing procedure described in Section \ref{sec:balancing}. However, these methods are not necessarily structured preserving and rely on a different type of Gramian $P_B$. In order to find a ROM for \eqref{bilinear_B}, state variables $x_2$ in  \eqref{sys:original_trans_B} need to be removed. These variables belong to the small HSVs $\sigma_{r_{B}+1},\ldots,\sigma_{n}$ and are hence negligible. A ROM \eqref{rom2} by BT is here obtained by truncating the second line of the state equation in \eqref{sys:original_trans_B} and to set $x_2(t) = 0$ in the remaining parts of the subsystem. This results in  \begin{align}\label{bt_B}
                                                                                                                                                                                                                                                                \left(\tilde A_{B}, \tilde B, \tilde C_{B}, \tilde D, \tilde E_{k},  \tilde N_{B, k}\right) = \left(A_{11}, B_{1}, C_{1}, 0, 0, N_{k, 11}\right).                                                                                                                                                                                                                                                \end{align}
Using similar arguments as in Section \ref{MOR_x0}, $\dot{x}_2(t)=0$ can be set alternatively in \eqref{sys:original_trans_B}. Additionally ignoring  the terms depending on ${N}_{k, 21}$ and ${N}_{k, 22}$, we obtain $x_2(t)=-A_{22}^{-1}\left(A_{21} x_1(t)+ B_2 u(t)\right)$. Inserting this for $x_2$ in \eqref{sys:original_trans_B}, a ROM \eqref{rom2} is obtained that has a different structure than \eqref{bilinear_B}. The associated matrices are  \begin{align}  \label{spa_B}                                                                                                                                                                                                                                                                                                                                                                                                                                                             
                                                                                                                                                                                                                                                              \left(\tilde A_{B}, \tilde B, \tilde C_{B}, \tilde D, \tilde E_{k},  \tilde N_{B, k}\right) = \left(\bar A, \bar B, \bar C, \bar D, \bar E_{k},  \bar {N}_{k}\right),                                                                                                                                                                                                                               \end{align}                                                                                                                                                                                                                                        where $\bar A :=A_{11}- A_{12} A_{22}^{-1} A_{21}$,  $\bar B:=B_1-A_{12}A_{22}^{-1} 
B_2$, $\bar C:=C_1-C_2 A_{22}^{-1} A_{21}$, $\bar D:=-C_2 A_{22}^{-1} B_2$, $\bar E_{k}:=-N_{k, 12} A_{22}^{-1} B_2$ and $\bar N_k:=N_{k, 11}-N_{k, 12} A_{22}^{-1} A_{21}$. It is important to mention that the main motivation behind the ROM given by \eqref{spa_B} is an error bound based on the sum of truncated HSVs that we state below. This shows the actual benefit of the structure change.   
\begin{remark}
Notice that both balancing related methods above preserve the type of stability given in \eqref{stab}. If $\tilde A_{B}$ and $\tilde N_{B, k}$ are as in  \eqref{bt_B} or \eqref{spa_B} and if additionally $\Sigma>0$ and $\sigma(\Sigma_1)\cap \sigma(\Sigma_2)=\emptyset$, we have \begin{align*}
 \sigma( \tilde A_{B}\otimes I+I\otimes  \tilde A_{B}+ \frac{1}{\gamma^2}\sum_{k=1}^m  \tilde N_{B, k} \otimes  \tilde N_{B, k})\subset \mathbb C_-.  \end{align*}
 This was proved in \citep[Theorem II.2]{bennerdammcruz} for BT and shown in \citep[Section 4.2]{redmannspa2} for SPA in the context of stochastic systems.
\end{remark}

\begin{thm}
 \label{special_error_bound2}
Let $y_{B}$ be the output of \eqref{bilinear_B}
given that \eqref{stab} holds for a sufficiently large $\gamma>0$. Suppose that $\tilde y_{B}$ is the reduced order output of system \eqref{rom2}, where the matrices $\left(\tilde A_{B}, \tilde B, \tilde C_{B}, \tilde D, \tilde E_{k},  \tilde N_{B, k}\right)$ are either given by BT stated in \eqref{bt_B} or by SPA defined in \eqref{spa_B}. Then, we have \begin{align*}
 \left\|y_{B} - \tilde y_{B} \right\|_{L^2} \leq   \left(2 \sum_{i= r_B+1}^n \sigma_i \right) \exp\left\{0.5\left\|\gamma u^{0}\right\|_{L^2}^2\right\} \left\|u\right\|_{L^2}, 
\end{align*}
where $\sigma_{r_B+1}, \ldots, \sigma_n$ are the truncated small HSVs of system \eqref{bilinear_B}.
\end{thm}
\begin{proof}
The above results for $\gamma=1$ are special cases of the ones in \citep[Theorem 3.1]{redstochbil} (BT) and  \citep[Theorem 3]{redmannstochbilspa} (SPA). Rescaling $N_k x_{B}(t) u_k(t)\mapsto  \frac{1}{\gamma} N_k x_{B}(t) \gamma u_k(t)$ in \eqref{bilinear_B_state} provides the formulation of this theorem for general $\gamma$.
\end{proof}
Theorem \ref{special_error_bound2} shows that the truncated HSVs determine the error of the approximation. Hence, these values are a good indicator for the expected error and can be chosen to find a suitable reduced order dimension $r_B$. Notice that for the HSVs it holds that $\sigma_i= \sigma_i(\gamma)$ since the underlying Gramians $P_B$ and $Q$ depend on the rescaling factor $\gamma$. 

\subsection{Model order reduction and an error bound for \eqref{fullsys}}

In this section, we exploit the results of Sections \ref{MOR_x0} and \ref{MOR_B} in order to find an error bound between the output $y$ of \eqref{fullsys} and some reduced output $\tilde y$ which we construct as the sum of the outputs $\tilde y_{x_0}$ and $\tilde y_{B}$ of subsystems \eqref{rom1} and \eqref{rom2}. We discussed BT and SPA for the homogeneous and the inhomogeneous part of the bilinear system in order to obtain $\tilde y_{x_0}$ and $\tilde y_{B}$. All approaches are designed to provide an error bound in $L^2$, which enables us to combine them in the following theorem.
\begin{thm}\label{main_result}
Suppose that \eqref{stab} holds for a sufficiently large $\gamma>0$. Let $y$ be the output of the original system \eqref{fullsys} and let us define the reduced output $\tilde y = \tilde y_{x_0}+\tilde y_{B}$, where $\tilde y_{x_0}$ is the quantity of interest in \eqref{rom1} and $\tilde y_{B}$ the one of \eqref{rom2}. We assume that \eqref{rom1} is the ROM of either BT stated in \eqref{bt_x0} or by SPA defined in \eqref{spa_x0} with $\Theta_1>0$, state dimension $r_{x_0}$ and an existing reduced Gramian $\tilde P$ for SPA. Let \eqref{rom2} be an $r_B$-dimensional ROM computed by BT with matrices \eqref{bt_B} or by SPA defined through \eqref{spa_B}.
Then, there exist a matrix $\mathcal W_{x_0}$ defined in Theorem \ref{special_error_bound} such that \begin{align*}
 \left\|y - \tilde y \right\|_{L^2} \leq  \left[ \left(\trace(\Theta_2 \mathcal W_{x_0})\right)^{\frac{1}{2}}  \left\|v_0\right\|_2+\left(2 \sum_{i= r_B+1}^n \sigma_i \right)  \left\|u\right\|_{L^2}\right]\exp\left\{0.5\left\|\gamma u^{0}\right\|_{L^2}^2\right\}
\end{align*}
with ${\Theta}_{2} = \diag(\theta_{r_{x_0}+1},\ldots,\theta_{n})$, 
where $\theta_{r_{x_0}+1}, \ldots, \theta_n$ and $\sigma_{r_B+1}, \ldots, \sigma_n$ are the truncated small HSVs of \eqref{bilinear_x0} and \eqref{bilinear_B}, respectively.
\end{thm}
\begin{proof}
 Let us recall that $y$ can be written as $y_{x_0}+y_B$, where $y_{x_0}$ is the output of \eqref{bilinear_x0} and $y_B$ the one of \eqref{bilinear_B}. Consequently, we have
 \begin{align*}
  \left\|y - \tilde y \right\|_{L^2}=\left\|y_{x_0}+y_B - (\tilde y_{x_0}+\tilde y_{B}) \right\|_{L^2}\leq \left\|y_{x_0}- \tilde y_{x_0} \right\|_{L^2} +\left\|y_B - \tilde y_{B} \right\|_{L^2}
 \end{align*}
using the triangle inequality. The claim now follows by Theorems \ref{special_error_bound} and \ref{special_error_bound2}.
\end{proof}
Theorem \ref{main_result} indicates that one finds a good approximation $\tilde y$ for the output $y$ of the large-scale system \eqref{fullsys} with non-zero initial states if each individual subsystem of \eqref{fullsys} is reduced such that the associated truncated HSVs are small. Depending on the decay, the number of truncated HSVs can differ in each subsystem leading to reduced order dimensions $r_{x_0}\neq r_B$.  The result of Theorem \ref{main_result} is the generalization of the error bound for the linear case proved in \citep[Theorem 3.2]{inhom_lin} if BT is applied to both subsystems \eqref{bilinear_x0} and \eqref{bilinear_B}. The result for the case of SPA as well as the combination of BT and SPA is new even for linear systems.\smallskip

The representation of the error bound nicely shows the relation between the error and the truncated HSVs which makes these values a good a-priori criterion for the choice of the reduced order dimensions. However, the first part of the bound is not suitable for practical computation as $\mathcal W_{x_0}$ depends on the full balanced realization \eqref{partmatricesx0} which is not computed in practice. Instead one can use the general representation $\mathcal E= \trace(CP_0 C^\top) +   \trace(\tilde C_{x_0} \tilde P \tilde C_{x_0}^\top) - 2 \trace(C \widehat P\tilde C_{x_0}^\top)$ from which $\trace(\Theta_2 \mathcal W_{x_0})$ was derived at the beginning of the proof of Theorem \ref{special_error_bound}. $\mathcal E$ is easily available since it involves the Gramian $P_0$ (which needs to be computed anyway to derive the reduced system) as well as the reduced Gramian $\tilde P$ and the solution  $\widehat P$ of \eqref{eq_mixed_gram} (both computationally cheap). Of course $\mathcal E$ then is an a-posteriori bound but still very powerful in order to determine an estimate for the reduction quality. The representation $\mathcal E$ provides another strategy in reducing \eqref{bilinear_x0} since $\sqrt{\mathcal E}$ is the $\mathcal H_2$-distance between systems \eqref{bilinear_B} and \eqref{rom2}, where $(B, \tilde B)$ are replaced by $(X_0, \tilde X_0)$, see \citep{morZhaL02}. Necessary conditions for local minimality have been provided in \citep{morZhaL02} and a method called bilinear iterative rational  Krylov algorithm (BIRKA) was proposed in \citep{breiten_benner} satisfying these conditions. Therefore, one can also use BIRKA instead of a balancing related scheme in order to keep the first summand of the bound in Theorem \ref{main_result} small.\smallskip

The second part of the bound in Theorem \ref{main_result} can be calculated once the Gramian $P_B$, satisfying the linear matrix inequality (LMI) \eqref{newgram2}, is computed. At the moment, LMI solver can only solve problems in moderate high dimensions, which might not be sufficient for some practical applications. However, once efficient computational methods are available, considering a Gramian like $P_B$ is very useful due to the a-priori $L^2$-error bound. In fact, only an $L^2$-bound is compatible with the bound in Theorem \ref{special_error_bound}. One might also think of choosing a Gramian like $P_0$ satisfying \eqref{reach_gram_bil_x0} for subsystem \eqref{bilinear_B} as proposed in \citep{typeIBT}. However, an $L^2$-error bound does not exist in such a case indicating the relevance of the new approach of choosing some LMI solution as a Gramian.
\begin{remark}\label{remark:gamma}
The value $\gamma>0$ was introduced in order to ensure \eqref{stab} which is a condition ensuring the existence of the Gramians. On the other hand, $\gamma$ can also be seen as an optimization parameter since the ROMs, and the HSVs depend on $\gamma$. This value was chosen equally in both subsystems, as one can see in Theorem \ref{main_result} but certainly they can also be different if this leads to a better reduction quality.
\end{remark}

We briefly summarize  this section by scheming the proposed MOR methods in Algorithm \ref{algo:SumUp}.

\begin{algorithm}[tb]
	\caption{MOR procedure for bilinear systems having non-zero initial conditions.}\label{algo:SumUp}
	\begin{algorithmic}[1]
		\Statex \textbf{Input:} System \eqref{fullsys} with matrices matrices $A$, $B$,  $N_k$, $X_0$ and $C$.
		\Statex \textbf{Output:} Two reduced systems of the form \eqref{rom1} and \eqref{rom2} such that $\tilde y_{x_0} + \tilde y_B$ approximates the original output $y$ of \eqref{fullsys}.
		\State Split original system \eqref{fullsys} into the subsystems \eqref{bilinear_x0} and \eqref{bilinear_B}.
		\State Compute  Gramians from  \eqref{reach_gram_bil_x0} and \eqref{obs_gram_bil_x0} corresponding to \eqref{bilinear_x0}. Determine associated Hankel singular values $\theta_i$ and decide for reduced order $r_{x_0}$ with $\theta_{r_{x_0}+1}, \dots, \theta_n$ small ensuring a low error according to Theorem \ref{special_error_bound}. 
		\State Compute Gramians  in \eqref{newgram2} and \eqref{obs_gram_bil_x0} corresponding to \eqref{bilinear_B}. Determine associated Hankel singular values $\sigma_i$ and decide for reduced order $r_B$  with $\sigma_{r_{B}+1}, \dots, \sigma_n$ small ensuring a low error according to Theorem \ref{special_error_bound2}. 
		\State Approximation of \eqref{bilinear_x0}:  Apply \texttt{BT} or \texttt{SPA} using the Gramians in  \eqref{reach_gram_bil_x0} and \eqref{obs_gram_bil_x0} to obtain a reduced order model of the form \eqref{rom1} (see Section \ref{MOR_x0}).
		\State Approximation of \eqref{bilinear_B}:  Apply \texttt{BT} or \texttt{SPA} using the Gramians in  \eqref{newgram2} and \eqref{obs_gram_bil_x0} to obtain a reduced order model of the form \eqref{rom2} (Section \ref{MOR_B}).
		\State (Optional) Compute $L^2$-error bound via Theorem \ref{main_result} to measure approximation quality.
	\end{algorithmic}
\end{algorithm}

\section{Numerical results}\label{sec:NumRes}
In this section, we conduct some numerical experiments illustrating the efficiency of the proposed MOR schemes.  All the simulations are done on a CPU 2,2 GHz  \intel~Core\texttrademark i7, 16 GB 2400 MHz DDR4, \matlab~9.1.0.441655 (R2016b).

We consider a standard test case model representing a  $2$D boundary controlled heat transfer system as described in \citep{bennerdamm}. The model is governed by the following boundary  value problem
\begin{align*}
	\begin{array}{rll}
		\partial_{t}\,X(t,\xi) &= \Delta X(t,\xi), & \text{if} \quad  \xi\in (0,1)\times(0,1),\quad \text{and} \quad t>0,    \\
		X(t,\xi) &= u(t), &  \text{if} \quad  \xi\in \Gamma_{1}, \\
		n\cdot \nabla X(t,\xi) &= u(t)X(t,\xi)  , & \text{if} \quad  \xi \in \Gamma_{3},\\
		X(t,\xi) &= 0, & \text{if} \quad  \xi \in \Gamma_{2}\cup\Gamma_{4}, \\
	\end{array}
\end{align*}
where $\Gamma_1 = \{0\}\times (0,1)$, $\Gamma_2 = (0,1)\times\{0\}$, $\Gamma_3 = \{1\}\times (0,1)$ and $\Gamma_4 = (0,1)\times\{1\}$. Here the heat transfer term $u$ acting on $\Gamma_1$ and $\Gamma_3$ is the input variable. Moreover, we assume that the initial temperature is positive and constant in space, i.e.,
\begin{align*}
X(0,\xi) = 0.1,  &  \quad  \xi\in (0,1)\times(0,1).
\end{align*} 
A semi-discretization in space using finite differences with $k = 10$ equidistant grid points in each direction leads to a bilinear  system of dimension $n = k^2 = 100$ having the form 
\begin{subequations}
\begin{align}\label{eq:NumExample}
	\dot{x}(t) &= Ax(t) + Nx(t)u(t) + Bu(t),\quad x(0) = X_0v_0, \\
	y(t) &= Cx(t),
\end{align}
\end{subequations}
where $X_0 = \begin{bmatrix}
	1 & \dots & 1
\end{bmatrix}^{\top}$, $v_0 = 0.1$ and $C = \frac{1}{n}\begin{bmatrix}
	1 & \dots & 1
\end{bmatrix}$ (see \citep{bennerdamm} for more details). 

 Firstly, in order to apply the proposed techniques, one need to compute $P_{0}$, $Q$ and $P_B$  satisfying equations \eqref{reach_gram_bil_x0}, \eqref{obs_gram_bil_x0} and \eqref{newgram2}, respectively. As shown in \citep{dammbennernewansatz}, by applying the Schur complement condition,  \eqref{newgram2} can be equivalently written as  the following linear matrix inequality
 \begin{equation}\label{lmi_formulation}
\begin{bmatrix}
AP_B +P_{B}A^\top +BB^\top & P_{B}N^\top 
\\
NP_{B} &    -P_{B}
\end{bmatrix} \leq 0.
 \end{equation}
Hence, we use the YALMIP toolbox \citep{lofberg2004yalmip} to the cost function $\trace(P_{B})$ subject to \eqref{lmi_formulation} in order to find a good candidate for the Gramian. 

Then, we compute the Hankel singular values associated to subsystems \eqref{bilinear_x0} and \eqref{bilinear_B} using, respectively, the pair of Gramians $(P_{x_0},Q)$ and $(P_{B},Q)$. The resulting Hankel singular values are depicted in Figure \ref{fig:HSV}.  We notice a fast decay of these  values, and hence, we expect accurate reduced models already for small reduced order dimensions as a consequence of Theorems \ref{special_error_bound} and \ref{special_error_bound2}.    
\begin{figure}[h]
	\newlength\wex
	\newlength\hex
	\setlength{\wex}{.8\textwidth}
	\setlength{\hex}{0.4\textwidth}
	\centering
	\includegraphics{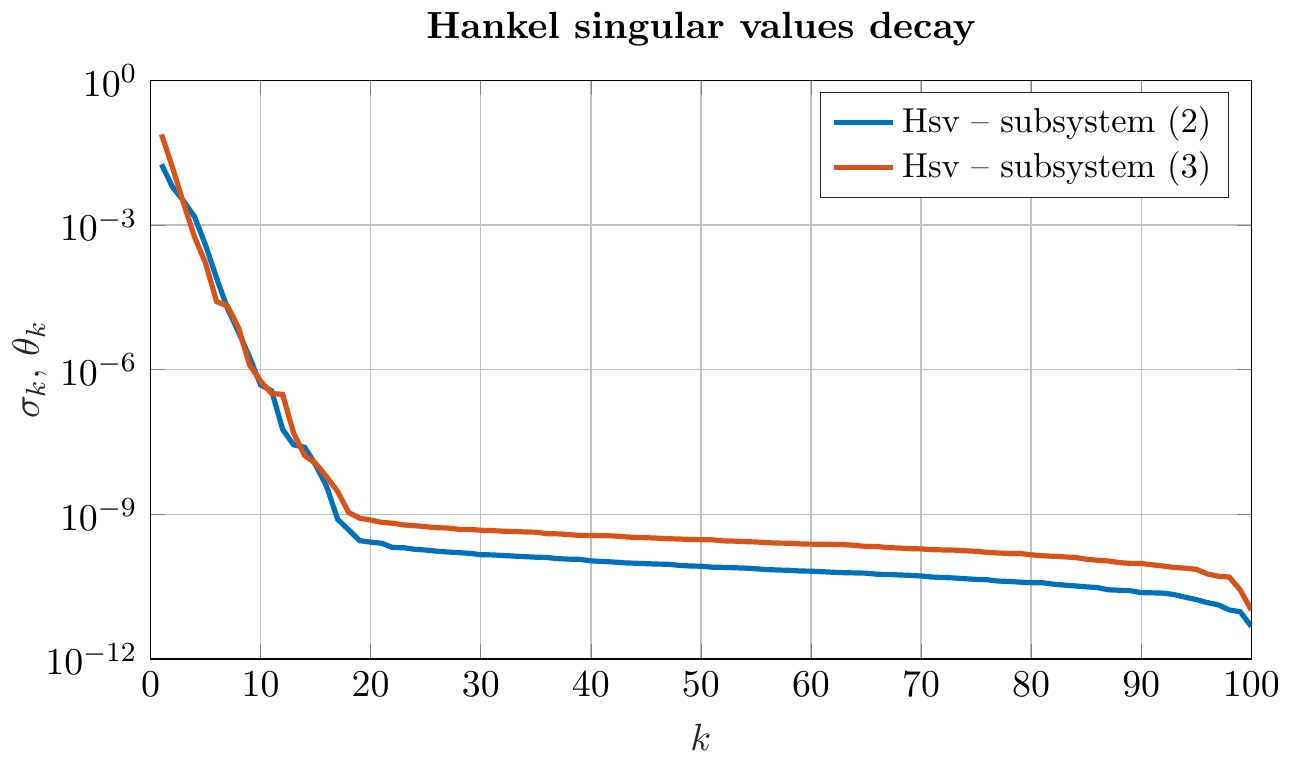}
	\caption{Decay of Hankel singular values for the subsystems \eqref{bilinear_x0} and \eqref{bilinear_B}. }\label{fig:HSV}
\end{figure}

For subsystem \eqref{bilinear_x0}, we construct ROMs of orders $5$ and $10$ using balanced truncation (referred here as \texttt{BT}) and SPA  (referred here as \texttt{SPA}) based on the Gramians $P_0$ and $Q$. Similarly, for subsystem \eqref{bilinear_B}, we construct  ROMs of order $5$ and $10$ using balanced truncation (referred here as \texttt{BT\_2}) and SPA (referred here as \texttt{SPA\_2}) based on the Gramians $P_B$ and $Q$. In order to compare the quality of ROMs, we simulate the original system and the reduced models using the input $u(t) = e^{-t}\cos(0.5t)$, $t\in [0, 1]$.  In Figure~\ref{fig:Err_x0}, the pointwise absolute errors for \texttt{BT} and \texttt{SPA} are depicted in semi-log scale, i.e., the function  $\epsilon(t) = |y(t) - \tilde{y}(t)|$ is computed, where $y$ and $\tilde{y}$ are, receptively, the original and reduced outputs. One can observe that the results improve once the reduced order is increased.  Both methodologies produce ROMs with a similar accuracy.  Similarly, in Figure~\ref{fig:Err_u} the pointwise absolute error plots for \texttt{BT\_2} and \texttt{SPA\_2}  are presented in semi-log scale. We notice that, for this example, \texttt{SPA\_2}  produces ROMs with a higher accuracy than \texttt{BT\_2}.  Additionally, in Tables  \ref{table:X0} and \ref{table:U}, the $L^2$-error values for the time interval $[0,1]$  together with the corresponding error bounds for the different methods are shown, respectively, for  subsystems \eqref{bilinear_x0} and \eqref{bilinear_B}.
\begin{table}[!tb]
	\centering
	\begin{tabular}{|c|c|c|c|c|}\hline
		\small Method & Err. bound $r_{x_0} =5$ & $L^2$-err. $r_{x_0}=5$ & Err. bound $r_{x_0} =10$ &  $L^2$-err. $r_{x_0}=10$  \\ \hline
		\texttt{\texttt{BT}} &      $7.64 \cdot 10^{-5}$      &  $2.80 \cdot 10^{-5}$            & $1.32\cdot 10^{-6}$ &      $6.76\cdot 10^{-7}$            \\ \hline
		\texttt{\texttt{SPA}}&   $1.53\cdot 10^{-4}$       & $5.33 \cdot 10^{-5}$             &$2.57\cdot 10^{-6}$ &   $1.77\cdot 10^{-6}$  \\ \hline            
	\end{tabular}
	\caption{$L^2$-errors for subsystem \eqref{bilinear_x0}: Error bounds in Theorem \ref{special_error_bound} and real errors for \texttt{BT} and \texttt{SPA} for the simulation presented in Figure \ref{fig:Err_x0}.}\label{table:X0}
\end{table}

\begin{table}[!tb]
	\centering
	\begin{tabular}{|c|c|c|c|c|}\hline
		\small Method & Err. bound $r_B =5$ & $L^2$-err. $r_B=5$ & Err. bound $r_B =10$ &  $L^2$-err. $r_B=10$  \\ \hline
		\texttt{\texttt{BT\_2}} &      $1.69\cdot 10^{-4}$      &   $3.25\cdot 10^{-5}$          & $7.88\cdot 10^{-7}$ &  $5.31\cdot 10^{-8} $              
		\\ \hline
		\texttt{\texttt{SPA\_2}}&  $1.69\cdot 10^{-4}$          &   $3.42\cdot 10^{-6}$          & $7.88\cdot 10^{-7}$ &     $4.00\cdot 10^{-9}$
		\\ \hline            
	\end{tabular}
	\caption{$L^2$-errors for subsystem \eqref{bilinear_B}: Error bound in Theorem \ref{special_error_bound2} and real errors for \texttt{BT\_2} and \texttt{SPA\_2} for the simulation presented in Figure \ref{fig:Err_u}.}\label{table:U}
\end{table}

\begin{figure}[h]
	\setlength{\wex}{.8\textwidth}
	\setlength{\hex}{0.4\textwidth}
	\centering
	\includegraphics{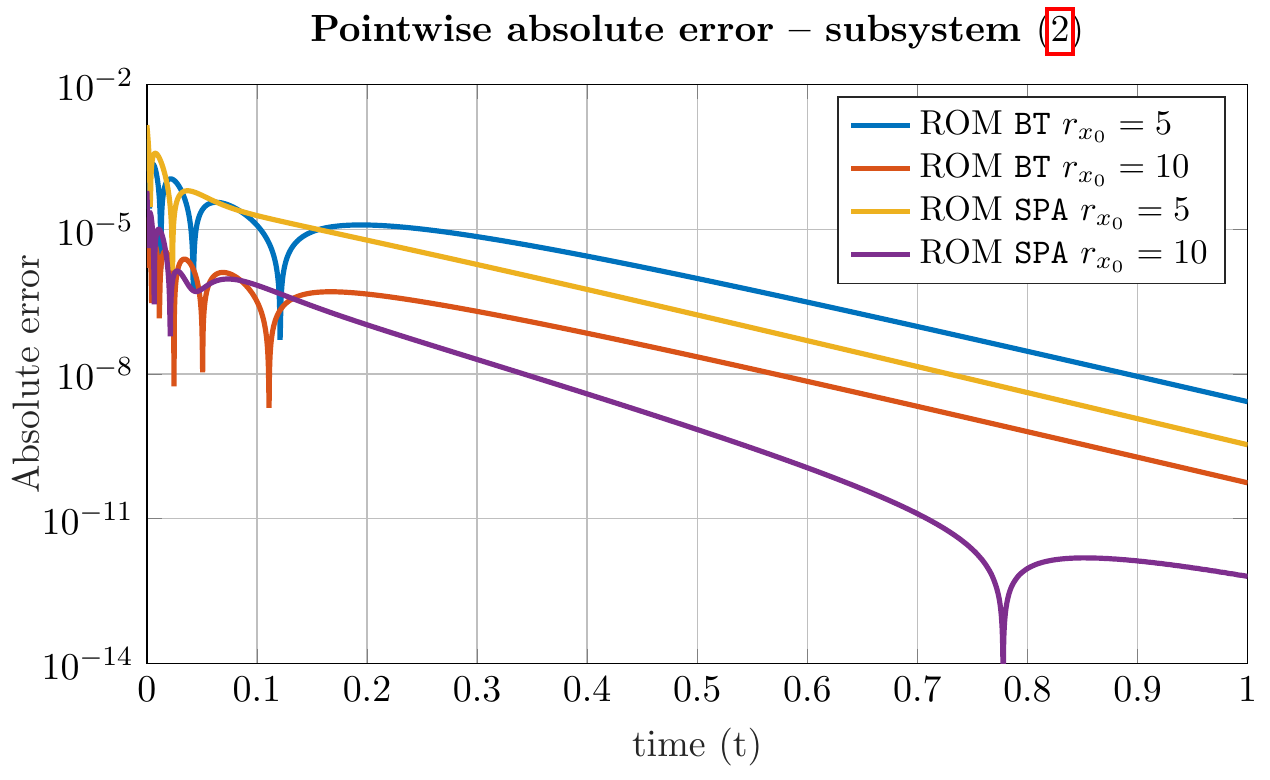}
	\caption{The pointwise absolute error for the approximations of subsystem \eqref{bilinear_x0} with input $u(t) = e^{-t}\cos(0.5t)$.
	}\label{fig:Err_x0}
\end{figure}

\begin{figure}[h]
	\setlength{\wex}{.8\textwidth}
	\setlength{\hex}{0.4\textwidth}
	\centering
	\includegraphics{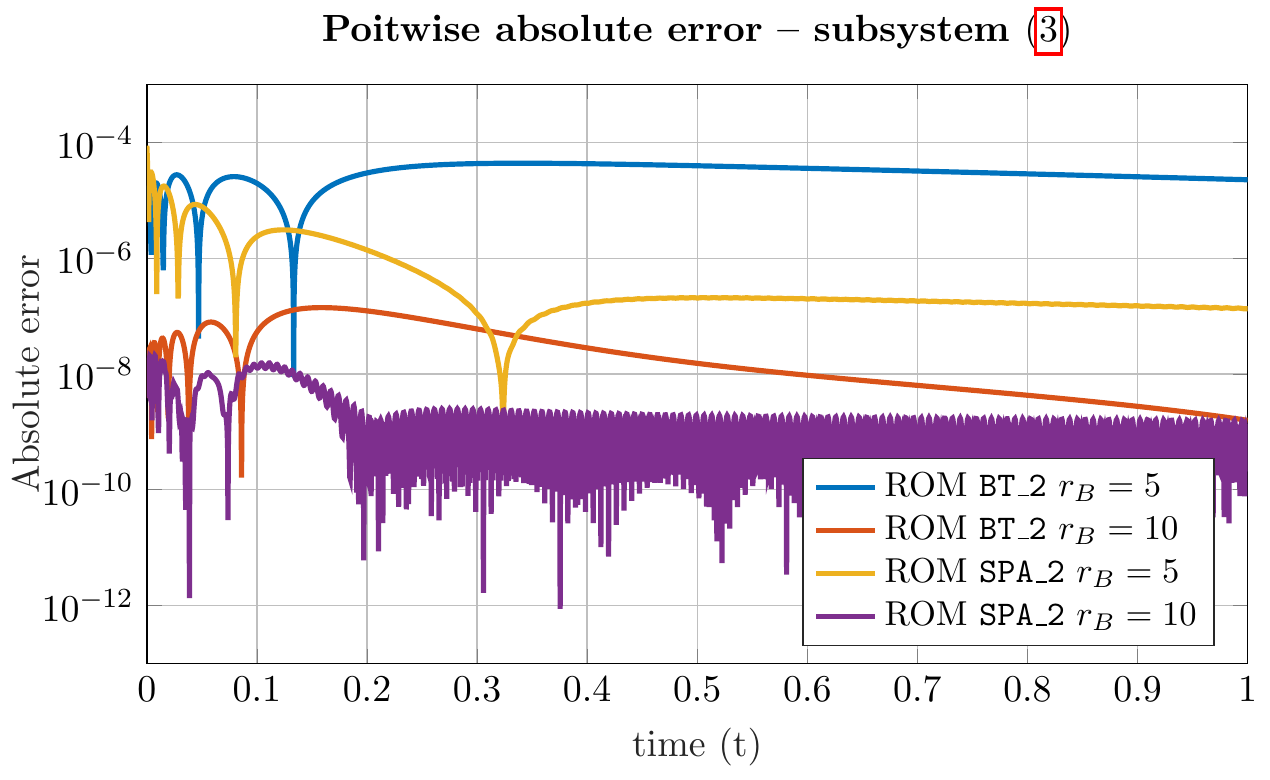}
	\caption{The pointwise absolute error for the approximations of subsystem \eqref{bilinear_B} with input $u(t) = e^{-t}\cos(0.5t)$.
	}\label{fig:Err_u}
\end{figure}

 In Figure \ref{fig:Err_vs_r_x0}, we depict the decay of  normalized $L^2$-error bounds from Theorem \ref{special_error_bound} together with the real $L^2$-errors produced by the time domain simulations for the subsystem \eqref{bilinear_x0} using the methods \texttt{BT} and \texttt{SPA}. Normalized here means to be divided by the $L^2$-norm of the original system output, e.g., the normalized $L^2$-error for a given order is  $\|y - \tilde{y}\|_{L^2}/\|y\|_{L^2}$, where $\tilde{y}$ is the reduced output. For this example, \texttt{BT} is performing slightly better than \texttt{SPA}.  Similarly, in Figure \ref{fig:Err_vs_r_u} we depict the decay of the normalized $L^2$-error bound from Theorem \ref{special_error_bound2} together with the normalized $L^2$-errors produced by the time domain simulations for the subsystem \eqref{bilinear_B} using the methods \texttt{BT\_2} and \texttt{SPA\_2}.  As stated before, \texttt{SPA\_2} produces ROMs that are more accurate than \texttt{BT\_2}.  

\begin{figure}[h]
	\setlength{\wex}{.8\textwidth}
	\setlength{\hex}{0.4\textwidth}
	\centering
	\includegraphics{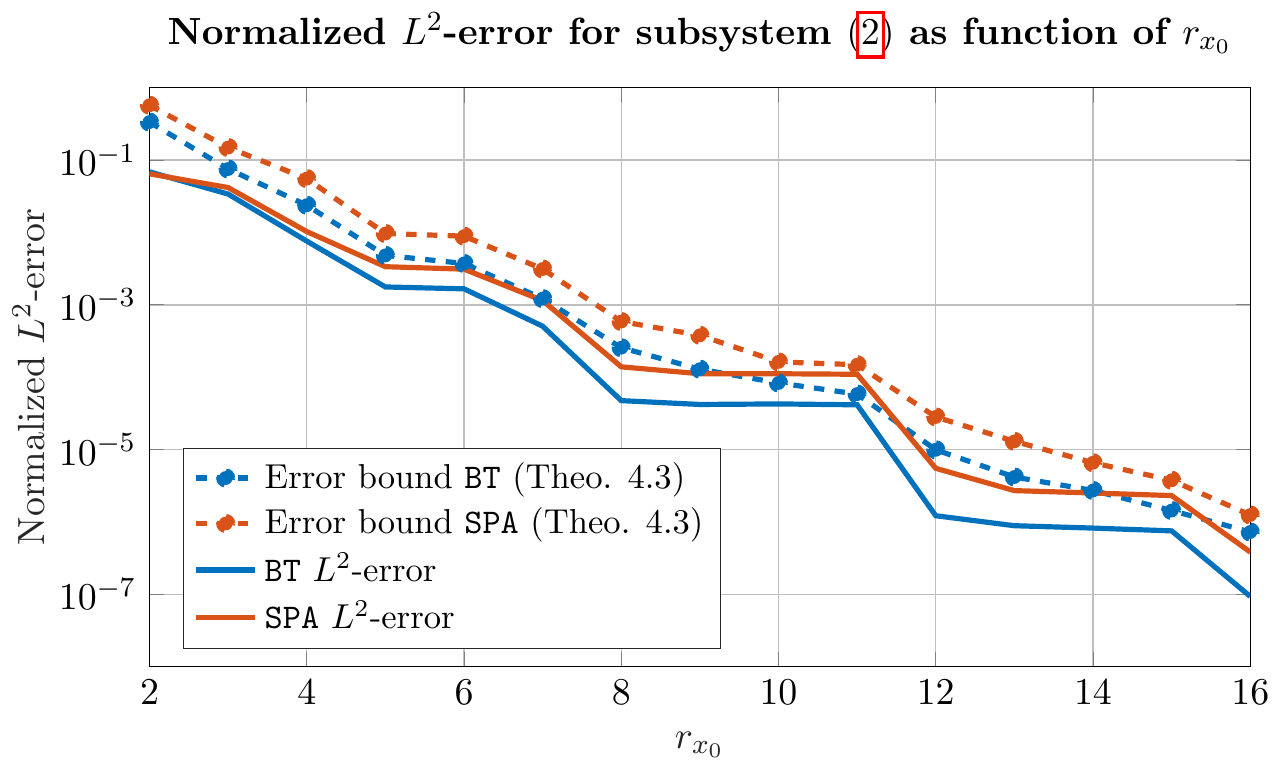}
	\caption{Subsystem \eqref{bilinear_x0}: Decay of normalized error bound ($\text{Err. Bound}/\|y_{x_0}\|_{L^2}$) and normalized $L^2$-error ($\|y_{x_0}-\tilde{y}_{x_0}\|_{L^2}/\|y_{x_0}\|_{L^2}$) produced by \texttt{SPA} and \texttt{BT} for different orders $r_{x_0} = 1, \dots, 16$.}\label{fig:Err_vs_r_x0}.
\end{figure}

\begin{figure}[h]
	\setlength{\wex}{.8\textwidth}
	\setlength{\hex}{0.4\textwidth}
	\centering
	\includegraphics{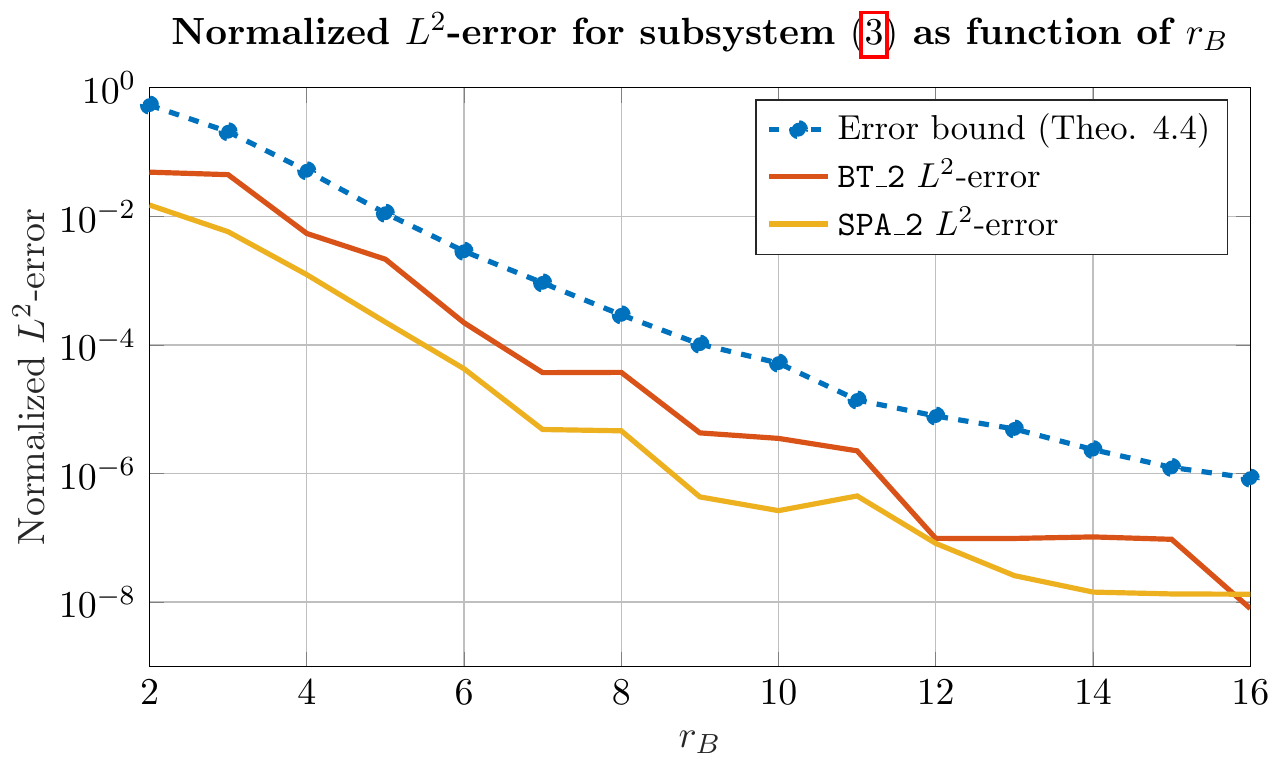}
	\caption{Subsystem \eqref{bilinear_B}: Decay of the normalized error bound ($\text{Err. Bound}/\|y_{B}\|_{L^2}$) and normalized $L^2$-error ($\|y_{B}-\tilde{y}_{B}\|_{L^2}/\|y_{B}\|_{L^2}$) produced by \texttt{SPA\_2} and \texttt{BT\_2} for different orders $r_B = 1, \dots, 16$. }\label{fig:Err_vs_r_u}.
\end{figure}

Finally, we illustrate the statement of Remark \ref{remark:gamma} and show the influence of the parameter  $\gamma$ on the approximation performance by means of Figure \ref{fig:Err_vs_gamma}. This figure depicts the error bound from Theorem \ref{special_error_bound} and the real $L^2$-error as a function of $\gamma$ (see Equations \eqref{reach_gram_bil_x0} and \eqref{obs_gram_bil_x0}) for a reduced dimension $r_{x_0} =10$. Furthermore, the simulations are conducted using the input $u(t) = e^{-t}\cos(0.5t)$ on the interval $[0,1]$.  We see that different values of $\gamma$ can lead to different qualities of the reduced models. Moreover, we observe numerically that for larger values $\gamma$, in this example $\gamma >6$, the real $L^2$-  error is increasing again. The choice of the optimal $\gamma$ is still an open problem, but not considered in this work. The numerical example suggests that one possible strategy is to select the $\gamma$ that minimizes the error bound of Theorem \ref{special_error_bound}.

\begin{figure}[h]
	\setlength{\wex}{.8\textwidth}
	\setlength{\hex}{0.4\textwidth}
	\centering
	\includegraphics{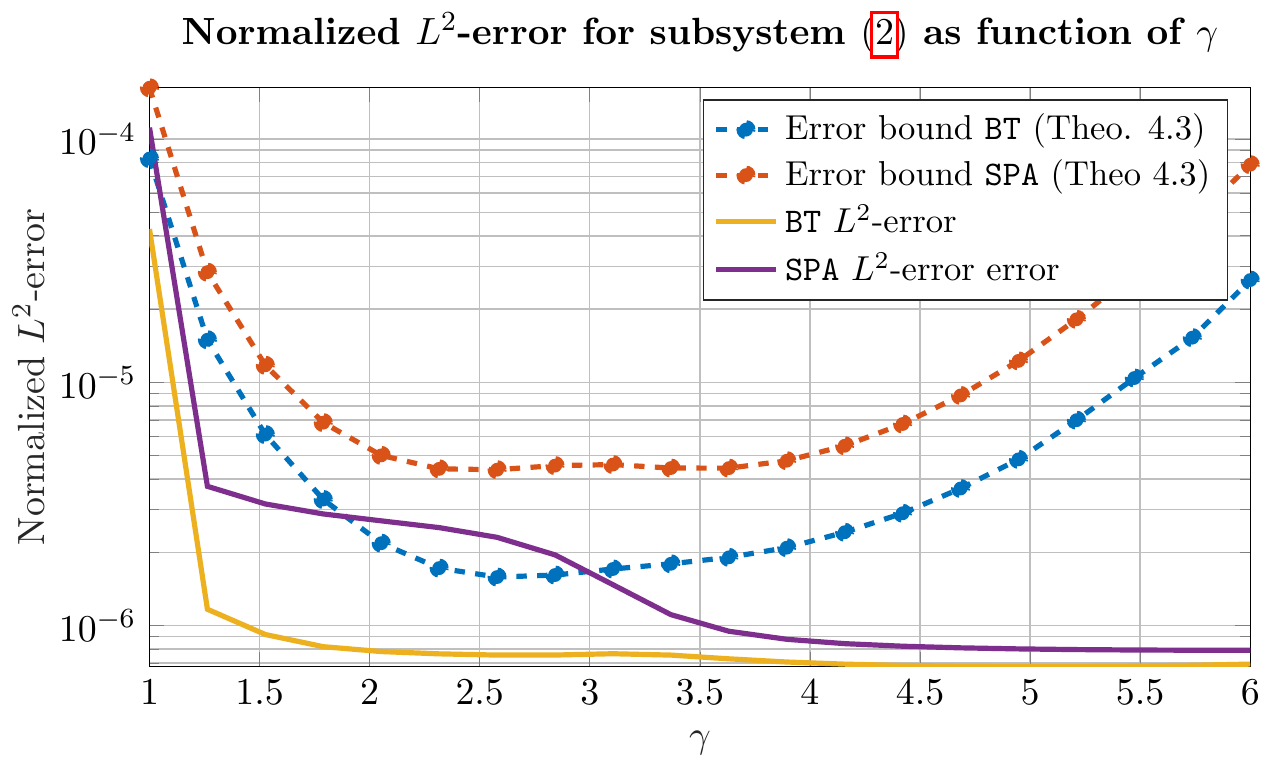}
	\caption{Subsystem \eqref{bilinear_x0}:  Normalized error bound ($\text{Err. Bound}/\|y_{x_0}\|_{L^2}$) and normalized $L^2$-error ($\|y_{x_0}-\tilde{y}_{x_0}\|_{L^2}/\|y_{x_0}\|_{L^2}$) produced by \texttt{SPA} and \texttt{BT} for different values $\gamma\in[1,6]$ and reduced dimension $r_{x_0} = 10$. }\label{fig:Err_vs_gamma}.
\end{figure}

%%%%%%%

\appendix

\section{Proofs of Section \ref{section2}}\label{appendix_section2}

\subsection{Proof of Lemma \ref{fund_est}}
\begin{proof}[Proof of Lemma \ref{fund_est}]
 We factorize $K= F F^\top$. 
Let $f_i$ be the $i$th column of the matrix $F$ and $x_{f_i}(\cdot, s)$ denote the solution to \eqref{bilinear_x0_state} with initial state $f_i$ and initial time $s$. Then, we have \begin{align*}
  \Phi(t, s) F  
  = [x_{f_1}(t, s), x_{f_2}(t, s), \ldots, x_{f_d}(t, s)],                                                                                \end{align*}
where $d$ is the number of columns of $F$. Using the scaling $\gamma>0$, $x_{f_i}(\cdot, s)$ can be interpreted as the solution to $\dot x_{f_i}(t) = Ax_{f_i}(t) +  \sum_{k=1}^m \frac{1}{\gamma} N_k x_{f_i}(t) \gamma u_k(t)$. Applying the results of \citep[Section 2]{h2_bil} on a bound for $x_{f_i}(t, s)x_{f_i}^\top(t, s)$, we obtain
  \begin{align*}
\Phi(t, s) K \Phi^\top(t, s) &= \sum_{k=1}^m x_{f_i}(t, s)x_{f_i}^\top(t, s) \leq  \exp\left\{\int_s^t \left\|\gamma u^{0}(v)\right\|_2^2 dv\right\} \sum_{k=1}^m Z_\gamma(t-s, f_i f_i^\top)\\&=\exp\left\{\int_s^t \left\|\gamma u^{0}(v)\right\|_2^2 dv\right\} Z_\gamma(t-s, K),
\end{align*}
where the second argument in $Z_\gamma$ denotes the respective initial condition. This concludes the proof.
\end{proof}

\subsection{Proof of Theorem \ref{cond_stab}}

Before we begin with the proof, we want to emphasize that it heavily makes use of relations between the bilinear equation \eqref{biode} and the stochastic system that results from setting $u_k = \frac{dw_k}{dt}$ in the bilinear term, where $\frac{dw_k}{dt}$ is white noise. Exploiting this link spares us from conducting relatively long and technical proofs. In particular, the stability analysis of bilinear and stochastic system is based on the same type of Lyapunov equations \citep{damm, redmannspa2}. Therefore, the first part of the Theorem \ref{cond_stab} follows from the result in \citep{redbendamm}. Moreover, the state of a bilinear system can be bounded by the state of the associated stochastic differential equation \citep{h2_bil}. This allows to transfer the result of \citep{martin_igor}, where the existence of a matrix $V$ as in Theorem \ref{cond_stab} was proved for stochastic systems. It turns out that the same holds true in the bilinear case using the same $V$ as in \citep{martin_igor}.
\begin{proof}[Proof of Theorem \ref{cond_stab}]
Condition \eqref{extended_stab} implies \eqref{bounded_Z} by \citep[Corollary 3.2]{redbendamm} or \citep[Lemma 6.12]{redmannPhD}. Now, we can use a stochastic representation for $Z_\gamma(\cdot, X_0 X_0^\top)$, see, e.g., \citep{damm, redmannspa2} which is $Z_\gamma(t, X_0 X_0^\top)= \mathbb E \left[\Phi_w(t)X_0 X_0^\top \Phi_w^\top(t)\right]$. Here, the stochastic fundamental solution $\Phi_w$ satisfies $\Phi_w(t) = I +\int_0^t A \Phi_w(s) ds + \sum_{k=1}^m \int_0^t N_k \Phi_w(s) dw_k(s)$ by definition, where $w_1, \ldots, w_m$ are independent standard Brownian motions. Based on \citep[Theorem 4.4, Remark 1]{martin_igor}, we then find \begin{align*}
  \left\|Z_\gamma(t, X_0 X_0^\top)\right\|_2\leq \mathbb E\left\|\Phi_w(t)X_0 X_0^\top \Phi_w^\top(t)\right\|_2\leq \mathbb E\left\|\Phi_w(t)X_0 \right\|_F^2  \lesssim \expn^{-c t}.                                                                                                                                                                                                                                                          
                                                                                                                                                                                                                                                                                                                                                                                                                                                                                               \end{align*}
Let us finally consider 
\begin{align*}
 \left\|\Phi(t)X_0 - V \tilde{\Phi}(t)\tilde X_0\right\|_F^2 &=\left\|\smat I & -V\srix\smat \Phi(t)& 0\\0 &\tilde{\Phi}(t)\srix \smat X_0\\ \tilde X_0\srix\right\|_F^2 \\&=\trace \left(\smat I & -V\srix\smat \Phi(t)& 0\\0 &\tilde{\Phi}(t)\srix \smat X_0\\ \tilde X_0\srix\smat X_0^\top & \tilde X_0^\top\srix \smat \Phi^\top(t)& 0\\0 &\tilde{\Phi}^\top(t)\srix \smat I \\ -V^\top\srix\right). 
\end{align*}
Since $\smat \Phi(t)& 0\\0 &\tilde{\Phi}(t)\srix $ is the fundamental solution to a bilinear system with matrices $\smat A& 0\\0 &\tilde A\srix $ and $\smat N_k& 0\\0 &\tilde N_k\srix $, we can apply Lemma \ref{fund_est} leading to \begin{align*}
 \left\|\Phi(t)X_0 - V \tilde{\Phi}(t)\tilde X_0\right\|_F^2 \leq \trace \left(\smat I & -V\srix Z_\gamma^e(t) \smat I \\ -V^\top\srix\right) \exp\left\{\int_0^t \left\|\gamma u^{0}(s)\right\|_2^2 ds\right\},
 \end{align*}
 where $Z_\gamma^e$ is the matrix function solving \eqref{eqZ} with coefficients $\smat A& 0\\0 &\tilde A\srix $,  $\smat N_k& 0\\0 &\tilde N_k\srix $ and $K=\smat X_0\\ \tilde X_0\srix\smat X_0^\top & \tilde X_0^\top\srix$. We exploit the associated stochastic representation which is $Z^e(t) = \mathbb E\left(\smat \Phi_w(t)& 0\\0 &\tilde{\Phi}_w(t)\srix \smat X_0\\ \tilde X_0\srix\smat X_0^\top & \tilde X_0^\top\srix\smat \Phi^\top_w(t)& 0\\0 &\tilde{\Phi}^\top_w(t)\srix\right) $, where $\tilde{\Phi}_w$ is the reduced order stochastic fundamental solution involving the matrices $\tilde A$ and $\tilde N_k$. Consequently, based on the linearity of the trace and the definition of the Frobenius norm, we have \begin{align*}
 \left\|\Phi(t)X_0 - V \tilde{\Phi}(t)\tilde X_0\right\|_F^2 \leq \mathbb E\left\|\Phi_w(t)X_0 - V \tilde{\Phi}_w(t)\tilde X_0\right\|_F^2 \exp\left\{\int_0^t \left\|\gamma u^{0}(s)\right\|_2^2 ds\right\}.
 \end{align*}
Due to \citep[Corollary 4.5, Remark 1]{martin_igor} we know about the existence of $V$ with $V^\top V=I$ such that $\Phi_w(t)X_0 = V \tilde{\Phi}_w(t)\tilde X_0$, where $\tilde{\Phi}_w$ decays exponentially in the mean square sense. This decay of $\tilde{\Phi}_w$ is equivalent to \eqref{stab_min_rel}, see, e.g., \citep{damm} which concludes the proof.
\end{proof}

\section{Proofs of Section \ref{MOR_x0}}\label{appendix_section4}

\subsection{Proof of Theorem \ref{red_gram_ex}}
\begin{proof}[Proof of Theorem \ref{red_gram_ex}]
 Since the Gramians of a balanced system are identical and equal to the diagonal matrix $\Theta$, we have \begin{align}\nonumber
&\smat{\mathcal A}_{11}&{\mathcal A}_{12}\\ 
{\mathcal A}_{21}&{\mathcal A}_{22}\srix \smat{\Theta}_{1}& \\ 
 &{\Theta}_{2}\srix+ \smat{\Theta}_{1}& \\ 
 &{\Theta}_{2}\srix \smat{\mathcal A}_{11}^\top&{\mathcal A}_{21}^\top\\ 
{\mathcal A}_{12}^\top&{\mathcal A}_{22}^\top\srix +\frac{1}{\gamma^2}\sum_{k=1}^{m}  \smat {\mathcal N}_{k, 11}&{\mathcal N}_{k, 12}\\ 
{\mathcal N}_{k, 21}&{\mathcal N}_{k, 22}\srix \smat{\Theta}_{1}& \\ 
 &{\Theta}_{2}\srix \smat {\mathcal N}_{k, 11}^\top&{\mathcal N}_{k, 21}^\top\\ 
{\mathcal N}_{k, 12}^\top&{\mathcal N}_{k, 22}^\top\srix \\ \label{bal_equation_x02}
&= -\smat X_{0, 1} \\ X_{0, 2}\srix \smat X_{0, 1}^\top &X_{0, 2}^\top\srix,\\ \nonumber
&\smat{\mathcal A}_{11}^\top&{\mathcal A}_{21}^\top\\ 
{\mathcal A}_{12}^\top&{\mathcal A}_{22}^\top\srix \smat{\Theta}_{1}& \\ 
 &{\Theta}_{2}\srix+ \smat{\Theta}_{1}& \\ 
 &{\Theta}_{2}\srix \smat{\mathcal A}_{11}&{\mathcal A}_{12}\\ 
{\mathcal A}_{21}&{\mathcal A}_{22}\srix +\frac{1}{\gamma^2}\sum_{k=1}^{m}  \smat {\mathcal N}_{k, 11}^\top&{\mathcal N}_{k, 21}^\top\\ 
{\mathcal N}_{k, 12}^\top&{\mathcal N}_{k, 22}^\top\srix \smat{\Theta}_{1}& \\ 
 &{\Theta}_{2}\srix \smat {\mathcal N}_{k, 11}&{\mathcal N}_{k, 12}\\ 
{\mathcal N}_{k, 21}&{\mathcal N}_{k, 22}\srix \\ \label{bal_equation_x0}
&= -\smat{\mathcal C}_1^\top \\{\mathcal C}_2^\top\srix \smat{\mathcal C}_1 &{\mathcal C}_2\srix.
    \end{align}  
The left upper blocks of these equations yield \begin{align}\label{romgrambt}
{\mathcal A}_{11}   \Theta_{1} + \Theta_{1}{\mathcal A}_{11}^\top + \frac{1}{\gamma^2}\sum_{k=1}^{m} {\mathcal N}_{k, 11}   \Theta_{1}  {\mathcal N}_{k, 11}^\top &\leq -  X_{0, 1} X_{0, 1}^\top,\\      
{\mathcal A}_{11}^\top   \Theta_{1} + \Theta_{1}{\mathcal A}_{11} + \frac{1}{\gamma^2}\sum_{k=1}^{m} {\mathcal N}_{k, 11}^\top   \Theta_{1}  {\mathcal N}_{k, 11} &\leq -  {\mathcal C}_1^\top {\mathcal C}_1.\nonumber
        \end{align}
Therefore, $\tilde Z_\gamma$ and $\tilde Z_\gamma^*$ decay exponentially by Theorem \ref{cond_stab} if BT is considered. Consequently, the integrals $\tilde P$ and $\tilde Q$ exist. We now multiply \eqref{bal_equation_x0} by \begin{align}\label{repinva}
\begin{bmatrix}
 {\mathcal A}_{11}&{\mathcal A}_{12}\\ 
{\mathcal A}_{21}&{\mathcal A}_{22}
\end{bmatrix}^{-1}
=\begin{bmatrix}
\bar {\mathcal A}^{-1}& \star\ {}\\
-{\mathcal A}_{22}^{-1}{\mathcal A}_{21}\bar {\mathcal A}^{-1}& \star\ {}
\end{bmatrix}
                  \end{align} 
from the right and with its transposed from the left. Evaluating the left upper block of the resulting equation and multiplying it with $\bar {\mathcal A}$ from the right and its transposed from the left, we find    \begin{align}\label{eq_red_spa}
\bar {\mathcal A}^\top   \Theta_{1} + \Theta_{1}\bar {\mathcal A} + \frac{1}{\gamma^2}\sum_{k=1}^{m} \bar {\mathcal N}_{k}^\top   \Theta_{1}  \bar{\mathcal N}_{k} = -  \bar {\mathcal C}^\top \bar {\mathcal C} - \frac{1}{\gamma^2}\sum_{k=1}^{m}\bar{\mathcal N}_{k, 21}^\top \Theta_2\bar{\mathcal N}_{k, 21}\leq -  \bar {\mathcal C}^\top \bar {\mathcal C},
        \end{align}
where $\bar{\mathcal N}_{k, 21}:={\mathcal N}_{k, 21}-{\mathcal N}_{k, 22}{\mathcal A}_{22}^{-1}{\mathcal A}_{21}$                    
providing the existence of $\tilde Q$ for SPA using Theorem \ref{cond_stab}.
\end{proof}

\subsection{Proof of Lemma \ref{basis_bound}}
\begin{proof}[Proof of Lemma \ref{basis_bound}]
 By Lemma \ref{lem_sol_rep}, we have that $y_{x_0}(t)= C \Phi(t) X_0 v_0$ and $y_{x_0}(t)= \tilde C_{x_0} \tilde \Phi(t) \tilde X_0 v_0$, where $\Phi$ and $\tilde \Phi$ are the fundamental solutions to the original and the reduced system, respectively, introduced in Definition \ref{defn_fund}. Consequently, we obtain \begin{align*}
 \left\|y_{x_0}(t) - \tilde y_{x_0}(t) \right\|_2^2 &\leq 
 \left\|C\Phi(t)X_0 - \tilde C_{x_0} \tilde{\Phi}(t)\tilde X_0\right\|_F^2 \left\|v_0\right\|_2^2 =\left\|\smat C & -\tilde C_{x_0}\srix\smat \Phi(t)& 0\\0 &\tilde{\Phi}(t)\srix \smat X_0\\ \tilde X_0\srix\right\|_F^2 \left\|v_0\right\|_2^2 \\
 &=\trace \left(\smat C & -\tilde C_{x_0}\srix\smat \Phi(t)& 0\\0 &\tilde{\Phi}(t)\srix \smat X_0\\ \tilde X_0\srix\smat X_0^\top & \tilde X_0^\top\srix \smat \Phi^\top(t)& 0\\0 &\tilde{\Phi}^\top(t)\srix \smat C^\top \\ -C_{x_0}^\top\srix\right)\left\|v_0\right\|_2^2. 
\end{align*}
Now, $\smat \Phi(t)& 0\\0 &\tilde{\Phi}(t)\srix$ is the fundamental solution to the bilinear system with matrices $\smat A& 0\\0 &\tilde A_{x_0}\srix $ and $\smat N_k& 0\\0 &\tilde N_{x_0, k}\srix$. Therefore, the result follows by Lemma \ref{fund_est} setting $K=\smat X_0\\ \tilde X_0\srix\smat X_0^\top & \tilde X_0^\top\srix$. 
\end{proof}

\subsection{Proof of Theorem \ref{special_error_bound}}
\begin{proof}[Proof of Theorem \ref{special_error_bound}]
 We integrate the result of Lemma \ref{basis_bound} on $[0, \infty)$ and obtain \begin{align*}
 \left\|y_{x_0} - \tilde y_{x_0} \right\|_{L^2}^2 \leq   \mathcal E \exp\left\{\left\|\gamma u^{0}\right\|_{L^2}^2\right\} \left\|v_0\right\|_2^2,                                                                                                                                   \end{align*}
where $\mathcal E:=\trace\left(\smat C & -\tilde C_{x_0}\srix \int_0^\infty Z_\gamma^e(t)dt \smat C^\top \\ -\tilde C_{x_0}^\top\srix\right)$. The left upper and the right lower block of $\int_0^\infty Z_\gamma^e(t)dt$ are $P_0$ and $\tilde P$, respectively. Both Gramians exist by assumption and Theorem \ref{red_gram_ex}. This also implies the existence of the right upper block of $\int_0^\infty Z_\gamma^e(t)dt$ which we denote by $\widehat P$. It satisfies \begin{align}  \label{eq_mixed_gram}                                                                                                                                                                                                                                                              
A \widehat P + \widehat P  \tilde A_{x_0}^\top + \frac{1}{\gamma^2}\sum_{k=1}^{m}  N_{k} \widehat P  \tilde N_{x_0, k}^\top   = -   X_0 X_{0, 1}^\top.                                                                                                                                                                                                                       \end{align}
Let $\mathcal S$ be the matrix ensuring \eqref{balanced_cal_S}. Since $\Theta = \mathcal S P_0 \mathcal S^\top$,  $Y= \mathcal S \widehat P$ and $\smat{\mathcal C}_1 &{\mathcal C}_2\srix = C\mathcal S^{-1} $, we have
\begin{align*}
 \mathcal E=& \trace(CP_0 C^\top) +   \trace(\tilde C_{x_0} \tilde P \tilde C_{x_0}^\top) - 2 \trace(C \widehat P\tilde C_{x_0}^\top) \\
 &= \trace(\smat{\mathcal C}_1 &{\mathcal C}_2\srix \smat{\Theta}_{1}& \\ 
 &{\Theta}_{2}\srix \smat{\mathcal C}_1^\top \\{\mathcal C}_2^\top\srix) +   \trace(\tilde C_{x_0} \tilde P_0 \tilde C_{x_0}^\top) - 2 \trace(\smat{\mathcal C}_1 &{\mathcal C}_2\srix Y \tilde C_{x_0}^\top).\end{align*}
 Comparing \eqref{bal_equation_x02} with \eqref{bal_equation_x0}, we see that $\trace(\smat{\mathcal C}_1 &{\mathcal C}_2\srix \Theta \smat{\mathcal C}_1^\top \\{\mathcal C}_2^\top\srix) = \trace(\smat{X}_{0, 1}^\top &{X}_{0, 2}^\top\srix \Theta \smat {X}_{0, 1} \\{X}_{0, 2}\srix)$. Since the same is true for the reduced Gramians, we obtain
 \begin{align}\label{firstEB_reform}
 \mathcal E= \trace(\smat{X}_{0, 1}^\top &{X}_{0, 2}^\top\srix \smat{\Theta}_{1}& \\ 
 &{\Theta}_{2}\srix \smat {X}_{0, 1} \\{X}_{0, 2}\srix) +   \trace({X}_{0, 1}^\top \tilde Q {X}_{0, 1}) - 2 \trace(\smat{\mathcal C}_1 &{\mathcal C}_2\srix Y \tilde C_{x_0}^\top),
\end{align}
where $\tilde Q$ exists due to Theorem \ref{red_gram_ex}. Now, it is needed to find an equation for $\tilde C_{x_0}^\top \smat{\mathcal C}_1 &{\mathcal C}\srix$ for both BT and SPA in order to analyze the error further.
We evaluate the first $r_{x_0}$ rows of \eqref{bal_equation_x0} to obtain an expression for the case of BT:\begin{align}\label{cceq_bt}
&-{\mathcal C}_1^\top \smat{\mathcal C}_1 &{\mathcal C}_2\srix =
\smat{\mathcal A}_{11}^\top {\Theta}_{1}&{\mathcal A}_{21}^\top{\Theta}_{2}\srix + \smat {\Theta}_{1}& 0\srix \smat{\mathcal A}_{11}&{\mathcal A}_{12}\\{\mathcal A}_{21}&{\mathcal A}_{22}\srix +\frac{1}{\gamma^2}\sum_{k=1}^{m}  \smat {\mathcal N}_{k, 11}^\top {\Theta}_{1} &{\mathcal N}_{k, 21}^\top {\Theta}_{2}\srix \smat {\mathcal N}_{k, 11}&{\mathcal N}_{k, 12}\\ 
{\mathcal N}_{k, 21}&{\mathcal N}_{k, 22}\srix.                                                                        \end{align}
For SPA we multiply \eqref{bal_equation_x0} with $\smat
 {\mathcal A}_{11}^\top&{\mathcal A}_{21}^\top\\ 
{\mathcal A}_{12}^\top&{\mathcal A}_{22}^\top
\srix^{-1}$ from the left and obtain
\begin{align*}
 -\smat {\bar {\mathcal A}}^{-\top} {\bar {\mathcal C}}^{\top} \\ \star \srix \smat{\mathcal C}_1 &{\mathcal C}_2\srix=&\smat{\Theta}_{1}& \\ 
 &{\Theta}_{2}\srix+ \smat {\bar {\mathcal A}}^{-\top}&-\bar {\mathcal A}^{-\top}({\mathcal A}_{22}^{-1}{\mathcal A}_{21})^\top\\ 
\star &\star \srix \smat{\Theta}_{1}& \\ 
 &{\Theta}_{2}\srix \smat
 {\mathcal A}_{11}&{\mathcal A}_{12}\\ 
{\mathcal A}_{21}&{\mathcal A}_{22}
\srix\\
&+\frac{1}{\gamma^2}\sum_{k=1}^{m}  \smat {\bar {\mathcal A}}^{-\top}{\bar {\mathcal N}}_{k}^\top& {\bar {\mathcal A}}^{-\top} \bar {\mathcal N}_{k, 21}^\top\\ 
\star&\star\srix \smat{\Theta}_{1}& \\ 
 &{\Theta}_{2}\srix \smat {\mathcal N}_{k, 11}&{\mathcal N}_{k, 12}\\ 
{\mathcal N}_{k, 21}&{\mathcal N}_{k, 22}\srix
    \end{align*}  
using the partition in \eqref{repinva} and setting $\bar{\mathcal N}_{k, 21}:={\mathcal N}_{k, 21}-{\mathcal N}_{k, 22}{\mathcal A}_{22}^{-1}{\mathcal A}_{21}$. Multiplying the first $r_{x_0}$ rows of the above equation by   $\bar {\mathcal A}^{\top}$ from the left results in 
\begin{align}\label{cceq_spa}
 -\bar {\mathcal C}^{\top} \smat{\mathcal C}_1 &{\mathcal C}_2\srix = \bar {\mathcal A}^{\top} {\Theta}_{1}+ \smat {\Theta_1} & -  ({\mathcal A}_{22}^{-1}{\mathcal A}_{21})^\top \Theta_2\srix \smat
 {\mathcal A}_{11}&{\mathcal A}_{12}\\ 
{\mathcal A}_{21}&{\mathcal A}_{22}
\srix +\frac{1}{\gamma^2}\sum_{k=1}^{m}  \smat {\bar {\mathcal N}}_{k}^\top {\Theta}_{1} &{\bar{\mathcal N}}_{k, 21}^\top {\Theta}_{2}\srix \smat {\mathcal N}_{k, 11}&{\mathcal N}_{k, 12}\\ 
{\mathcal N}_{k, 21}&{\mathcal N}_{k, 22}\srix.
\end{align}
We summarize \eqref{cceq_bt} and \eqref{cceq_spa} to one equation. That is \begin{align*}
 -\tilde C_{x_0}^{\top} \smat{\mathcal C}_1 &{\mathcal C}_2\srix = \smat {\tilde A}_{x_0}^{\top} {\Theta}_{1}& {\mathbf A}_{21}^\top {\Theta}_{2} \srix+ \smat {\Theta_1} & {\bar{\mathbf A}}_{21}^\top \Theta_2\srix \smat
 {\mathcal A}_{11}&{\mathcal A}_{12}\\ 
{\mathcal A}_{21}&{\mathcal A}_{22}
\srix +\frac{1}{\gamma^2}\sum_{k=1}^{m}  \smat {\tilde N}_{x_0, k}^\top {\Theta}_{1} &{\mathbf N}_{k, 21}^\top {\Theta}_{2}\srix \smat {\mathcal N}_{k, 11}&{\mathcal N}_{k, 12}\\ 
{\mathcal N}_{k, 21}&{\mathcal N}_{k, 22}\srix.
\end{align*}
where ${\mathbf A}_{21}\in\{\mathcal A_{21}, 0\}$, $\bar {\mathbf A}_{21}\in \{0, -  {\mathcal A}_{22}^{-1}{\mathcal A}_{21}\}$ and ${\mathbf N}_{k, 21}\in \{\mathcal N_{k, 21}, {\bar {\mathcal N}}_{k, 21}\}$. Inserting this into $\trace(\smat{\mathcal C}_1 &{\mathcal C}_2\srix Y \tilde C_{x_0}^{\top})$ yields \begin{align*}
&-\trace(\smat{\mathcal C}_1 &{\mathcal C}_2\srix Y \tilde C_{x_0}^{\top})\\
&=\trace\left(Y\left[\smat {\tilde A}_{x_0}^{\top} {\Theta}_{1}& {\mathbf A}_{21}^\top {\Theta}_{2} \srix+ \smat {\Theta_1} & {\bar{\mathbf A}}_{21}^\top \Theta_2\srix \smat
 {\mathcal A}_{11}&{\mathcal A}_{12}\\ 
{\mathcal A}_{21}&{\mathcal A}_{22}
\srix +\frac{1}{\gamma^2}\sum_{k=1}^{m}  \smat {\tilde N}_{x_0, k}^\top {\Theta}_{1} &{\mathbf N}_{k, 21}^\top {\Theta}_{2}\srix \smat {\mathcal N}_{k, 11}&{\mathcal N}_{k, 12}\\ 
{\mathcal N}_{k, 21}&{\mathcal N}_{k, 22}\srix\right]\right)\\
&=   \trace\left(\Theta_1\left[ \smat{\mathcal A}_{11}&{\mathcal A}_{12}\srix Y +Y_1 {\tilde A}_{x_0}^\top + \frac{1}{\gamma^2}\sum_{k=1}^{m}    
 \smat {\mathcal N}_{k, 11}&{\mathcal N}_{k, 12}\srix Y {\tilde N}_{x_0, k}^\top \right]\right) \\
&\quad + \trace\left(\Theta_2\left[ Y_2{\mathbf A}_{21}^\top + \smat{\mathcal A}_{21}&{\mathcal A}_{22}\srix Y {\bar{\mathbf A}}_{21}^\top+ \frac{1}{\gamma^2}\sum_{k=1}^{m} \smat
{\mathcal N}_{k, 21}&{\mathcal N}_{k, 22}\srix Y{\mathbf N}_{k, 21}^\top\right]\right).
                      \end{align*}
The first $r_{x_0}$ rows of \eqref{yequation} give us \begin{align*}
 -\trace(\smat{\mathcal C}_1 &{\mathcal C}_2\srix Y \mathcal  C_{1}^\top) = &-\trace\left(X_{0, 1}^\top\Theta_1 X_{0, 1}\right) \\
 &+ \trace\left(\Theta_2\left[ Y_2{\mathbf A}_{21}^\top + \smat{\mathcal A}_{21}&{\mathcal A}_{22}\srix Y {\bar{\mathbf A}}_{21}^\top+ \frac{1}{\gamma^2}\sum_{k=1}^{m} \smat
{\mathcal N}_{k, 21}&{\mathcal N}_{k, 22}\srix Y{\mathbf N}_{k, 21}^\top\right]\right).
\end{align*}
Inserting this into \eqref{firstEB_reform} leads to \begin{align*}
 \mathcal E = &\trace(X_{0, 1}^\top(\tilde Q-\Theta_1)X_{0, 1}) \\
 &+ \trace\left(\Theta_2\left[  X_{0, 2} X_{0, 2}^\top + 2 Y_2{\mathbf A}_{21}^\top + 2\smat{\mathcal A}_{21}&{\mathcal A}_{22}\srix Y {\bar{\mathbf A}}_{21}^\top+ \frac{2}{\gamma^2}\sum_{k=1}^{m} \smat
{\mathcal N}_{k, 21}&{\mathcal N}_{k, 22}\srix Y{\mathbf N}_{k, 21}^\top\right]\right),
                                                          \end{align*}
By the left upper $r_{x_0}\times r_{x_0}$ block of \eqref{bal_equation_x0} (BT) and \eqref{eq_red_spa} (SPA), it holds that \begin{align*}                                                                                                                
{\tilde A}_{x_0}^{\top} \Theta_1 + \Theta_1 {\tilde A}_{x_0}+ \frac{1}{\gamma^2}\sum_{k=1}^{m} {\tilde N}_{x_0, k}^{\top} \Theta_1 {\tilde N}_{x_0, k}  =- {\tilde C}_{x_0}^{\top} {\tilde C}_{x_0}  -\frac{1}{\gamma^2}\sum_{k=1}^{m} \mathbf N_{k, 21}^\top \Theta_2 \mathbf N_{k, 21}                                                                                                                         \end{align*}
for both reduced order schemes. Subtracting this identity from the equation for the reduced Gramian $\tilde Q$, we find 
\begin{align*}
 {\tilde A}_{x_0}^{\top} (\tilde Q-\Theta_1) + (\tilde Q-\Theta_1) {\tilde A}_{x_0}+ \frac{1}{\gamma^2}\sum_{k=1}^{m} {\tilde N}_{x_0, k}^{\top} (\tilde Q-\Theta_1) {\tilde N}_{x_0, k}  = \frac{1}{\gamma^2}\sum_{k=1}^{m} \mathbf N_{k, 21}^\top \Theta_2 \mathbf N_{k, 21}.
\end{align*}
Hence, exploiting the equation for $\tilde P$, we have \begin{align*}
&\trace(X_{0, 1}^\top(\tilde Q-\Theta_1)X_{0, 1}) = - \trace\left(\left[{\tilde A}_{x_0}\tilde P+ \tilde P {\tilde A}_{x_0}^{\top}+\frac{1}{\gamma^2}\sum_{k=1}^{m} {\tilde N}_{x_0, k} \tilde P {\tilde N}_{x_0, k}^{\top} \right] (\tilde Q-\Theta_1)\right)\\
&= - \trace\left(\tilde P\left[{\tilde A}_{x_0}^\top (\tilde Q-\Theta_1)\ + (\tilde Q-\Theta_1)\ {\tilde A}_{x_0}+ \frac{1}{\gamma^2}\sum_{k=1}^{m} {\tilde N}_{x_0, k}^\top (\tilde Q-\Theta_1)\ {\tilde N}_{x_0, k} \right]\right)\\
&= - \trace\left( \Theta_2 \frac{1}{\gamma^2}\sum_{k=1}^{m}  \mathbf N_{k, 21}\tilde P \mathbf N_{k, 21}^\top\right),
               \end{align*}
which concludes the proof of this theorem.
\end{proof}

\bibliographystyle{apacite}
%\bibliography{refs}

\begin{thebibliography}{}

\bibitem [\protect \citeauthoryear {%
Al-Baiyat%
\ \BBA {} Bettayeb%
}{%
Al-Baiyat%
\ \BBA {} Bettayeb%
}{%
{\protect \APACyear {1993}}%
}]{%
typeIBT}
\APACinsertmetastar {%
typeIBT}%
\begin{APACrefauthors}%
Al-Baiyat, S\BPBI A.%
\BCBT {}\ \BBA {} Bettayeb, M.%
\end{APACrefauthors}%
\unskip\
\newblock
\APACrefYearMonthDay{1993}{}{}.
\newblock
{\BBOQ}\APACrefatitle {{A new model reduction scheme for k--power bilinear
  systems}} {{A new model reduction scheme for k--power bilinear
  systems}}.{\BBCQ}
\newblock
\APACjournalVolNumPages{Proceedings of the 32nd IEEE Conference on Decision and
  Control}{}{}{22--27}.
\PrintBackRefs{\CurrentBib}

\bibitem [\protect \citeauthoryear {%
Antoulas%
, Gosea%
\BCBL {}\ \BBA {} Ionita%
}{%
Antoulas%
\ \protect \BOthers {.}}{%
{\protect \APACyear {2016}}%
}]{%
loewener_bil}
\APACinsertmetastar {%
loewener_bil}%
\begin{APACrefauthors}%
Antoulas, A\BPBI C.%
, Gosea, I\BPBI V.%
\BCBL {}\ \BBA {} Ionita, A\BPBI C.%
\end{APACrefauthors}%
\unskip\
\newblock
\APACrefYearMonthDay{2016}{}{}.
\newblock
{\BBOQ}\APACrefatitle {{Model Reduction of Bilinear Systems in the Loewner
  Framework}} {{Model Reduction of Bilinear Systems in the Loewner
  Framework}}.{\BBCQ}
\newblock
\APACjournalVolNumPages{SIAM J. Sci. Comput.}{38}{5}{B889--B916}.
\PrintBackRefs{\CurrentBib}

\bibitem [\protect \citeauthoryear {%
Bauer%
, Benner%
\BCBL {}\ \BBA {} Feng%
}{%
Bauer%
\ \protect \BOthers {.}}{%
{\protect \APACyear {2014}}%
}]{%
BauerBenner_inhom}
\APACinsertmetastar {%
BauerBenner_inhom}%
\begin{APACrefauthors}%
Bauer, U.%
, Benner, P.%
\BCBL {}\ \BBA {} Feng, L.%
\end{APACrefauthors}%
\unskip\
\newblock
\APACrefYearMonthDay{2014}{}{}.
\newblock
{\BBOQ}\APACrefatitle {{Model Order Reduction for Linear and Nonlinear Systems:
  A System-Theoretic Perspective}} {{Model Order Reduction for Linear and
  Nonlinear Systems: A System-Theoretic Perspective}}.{\BBCQ}
\newblock
\APACjournalVolNumPages{{Arch. Comput. Methods Eng.}}{21}{4}{331--358}.
\PrintBackRefs{\CurrentBib}

\bibitem [\protect \citeauthoryear {%
Beattie%
, Gugercin%
\BCBL {}\ \BBA {} Mehrmann%
}{%
Beattie%
\ \protect \BOthers {.}}{%
{\protect \APACyear {2017}}%
}]{%
inhom_lin}
\APACinsertmetastar {%
inhom_lin}%
\begin{APACrefauthors}%
Beattie, C.%
, Gugercin, S.%
\BCBL {}\ \BBA {} Mehrmann, V.%
\end{APACrefauthors}%
\unskip\
\newblock
\APACrefYearMonthDay{2017}{}{}.
\newblock
{\BBOQ}\APACrefatitle {{Model reduction for systems with inhomogeneous initial
  conditions}} {{Model reduction for systems with inhomogeneous initial
  conditions}}.{\BBCQ}
\newblock
\APACjournalVolNumPages{{Syst. Control. Lett.}}{99}{}{99--106}.
\PrintBackRefs{\CurrentBib}

\bibitem [\protect \citeauthoryear {%
Benner%
\ \BBA {} Breiten%
}{%
Benner%
\ \BBA {} Breiten%
}{%
{\protect \APACyear {2012}}%
}]{%
breiten_benner}
\APACinsertmetastar {%
breiten_benner}%
\begin{APACrefauthors}%
Benner, P.%
\BCBT {}\ \BBA {} Breiten, T.%
\end{APACrefauthors}%
\unskip\
\newblock
\APACrefYearMonthDay{2012}{}{}.
\newblock
{\BBOQ}\APACrefatitle {{Interpolation-based $\mathcal{H}_2$-model reduction of
  bilinear control systems}} {{Interpolation-based $\mathcal{H}_2$-model
  reduction of bilinear control systems}}.{\BBCQ}
\newblock
\APACjournalVolNumPages{SIAM J. Matrix Anal. Appl.}{33}{3}{859--885}.
\PrintBackRefs{\CurrentBib}

\bibitem [\protect \citeauthoryear {%
Benner%
\ \BBA {} Damm%
}{%
Benner%
\ \BBA {} Damm%
}{%
{\protect \APACyear {2011}}%
}]{%
bennerdamm}
\APACinsertmetastar {%
bennerdamm}%
\begin{APACrefauthors}%
Benner, P.%
\BCBT {}\ \BBA {} Damm, T.%
\end{APACrefauthors}%
\unskip\
\newblock
\APACrefYearMonthDay{2011}{}{}.
\newblock
{\BBOQ}\APACrefatitle {{Lyapunov equations, energy functionals, and model order
  reduction of bilinear and stochastic systems.}} {{Lyapunov equations, energy
  functionals, and model order reduction of bilinear and stochastic
  systems.}}{\BBCQ}
\newblock
\APACjournalVolNumPages{SIAM J. Control Optim.}{49}{2}{686--711}.
\newblock
\begin{APACrefDOI} \doi{10.1137/09075041X} \end{APACrefDOI}
\PrintBackRefs{\CurrentBib}

\bibitem [\protect \citeauthoryear {%
Benner%
, Damm%
, Redmann%
\BCBL {}\ \BBA {} Rodriguez~Cruz%
}{%
Benner%
\ \protect \BOthers {.}}{%
{\protect \APACyear {2016}}%
}]{%
redbendamm}
\APACinsertmetastar {%
redbendamm}%
\begin{APACrefauthors}%
Benner, P.%
, Damm, T.%
, Redmann, M.%
\BCBL {}\ \BBA {} Rodriguez~Cruz, Y\BPBI R.%
\end{APACrefauthors}%
\unskip\
\newblock
\APACrefYearMonthDay{2016}{}{}.
\newblock
{\BBOQ}\APACrefatitle {{Positive Operators and Stable Truncation}} {{Positive
  Operators and Stable Truncation}}.{\BBCQ}
\newblock
\APACjournalVolNumPages{Linear Algebra Appl}{498}{}{74--87}.
\PrintBackRefs{\CurrentBib}

\bibitem [\protect \citeauthoryear {%
Benner%
, Damm%
\BCBL {}\ \BBA {} Rodriguez~Cruz%
}{%
Benner%
\ \protect \BOthers {.}}{%
{\protect \APACyear {2017}}%
}]{%
bennerdammcruz}
\APACinsertmetastar {%
bennerdammcruz}%
\begin{APACrefauthors}%
Benner, P.%
, Damm, T.%
\BCBL {}\ \BBA {} Rodriguez~Cruz, Y\BPBI R.%
\end{APACrefauthors}%
\unskip\
\newblock
\APACrefYearMonthDay{2017}{}{}.
\newblock
{\BBOQ}\APACrefatitle {{Dual pairs of generalized Lyapunov inequalities and
  balanced truncation of stochastic linear systems}} {{Dual pairs of
  generalized Lyapunov inequalities and balanced truncation of stochastic
  linear systems}}.{\BBCQ}
\newblock
\APACjournalVolNumPages{IEEE Trans. Autom. Contr.}{62}{2}{782--791}.
\PrintBackRefs{\CurrentBib}

\bibitem [\protect \citeauthoryear {%
Cao%
, Benner%
, Pontes~Duff%
\BCBL {}\ \BBA {} Schilders%
}{%
Cao%
\ \protect \BOthers {.}}{%
{\protect \APACyear {2020}}%
}]{%
inhom_bil}
\APACinsertmetastar {%
inhom_bil}%
\begin{APACrefauthors}%
Cao, X.%
, Benner, P.%
, Pontes~Duff, I.%
\BCBL {}\ \BBA {} Schilders, W.%
\end{APACrefauthors}%
\unskip\
\newblock
\APACrefYearMonthDay{2020}{}{}.
\newblock
{\BBOQ}\APACrefatitle {{Model order reduction for bilinear control systems with
  inhomogeneous initial conditions}} {{Model order reduction for bilinear
  control systems with inhomogeneous initial conditions}}.{\BBCQ}
\newblock
\APACjournalVolNumPages{Int. J. Control}{}{}{}.
\PrintBackRefs{\CurrentBib}

\bibitem [\protect \citeauthoryear {%
Damm%
}{%
Damm%
}{%
{\protect \APACyear {2004}}%
}]{%
damm}
\APACinsertmetastar {%
damm}%
\begin{APACrefauthors}%
Damm, T.%
\end{APACrefauthors}%
\unskip\
\newblock
\APACrefYear{2004}.
\newblock
\APACrefbtitle {{Rational Matrix Equations in Stochastic Control.}} {{Rational
  Matrix Equations in Stochastic Control.}}
\newblock
\APACaddressPublisher{}{{Lecture Notes in Control and Information Sciences 297.
  Berlin: Springer}}.
\PrintBackRefs{\CurrentBib}

\bibitem [\protect \citeauthoryear {%
Damm%
\ \BBA {} Benner%
}{%
Damm%
\ \BBA {} Benner%
}{%
{\protect \APACyear {2014}}%
}]{%
dammbennernewansatz}
\APACinsertmetastar {%
dammbennernewansatz}%
\begin{APACrefauthors}%
Damm, T.%
\BCBT {}\ \BBA {} Benner, P.%
\end{APACrefauthors}%
\unskip\
\newblock
\APACrefYearMonthDay{2014}{}{}.
\newblock
{\BBOQ}\APACrefatitle {{Balanced truncation for stochastic linear systems with
  guaranteed error bound}} {{Balanced truncation for stochastic linear systems
  with guaranteed error bound}}.{\BBCQ}
\newblock
\APACjournalVolNumPages{Proceedings of MTNS--2014, Groningen, The
  Netherlands}{}{}{1492--1497}.
\PrintBackRefs{\CurrentBib}

\bibitem [\protect \citeauthoryear {%
Daragmeh%
, Hartmann%
\BCBL {}\ \BBA {} Qatanani%
}{%
Daragmeh%
\ \protect \BOthers {.}}{%
{\protect \APACyear {2019}}%
}]{%
spa_inhom}
\APACinsertmetastar {%
spa_inhom}%
\begin{APACrefauthors}%
Daragmeh, A.%
, Hartmann, C.%
\BCBL {}\ \BBA {} Qatanani, N.%
\end{APACrefauthors}%
\unskip\
\newblock
\APACrefYearMonthDay{2019}{}{}.
\newblock
{\BBOQ}\APACrefatitle {{Balanced model reduction of linear systems with nonzero
  initial conditions: Singular perturbation approximation}} {{Balanced model
  reduction of linear systems with nonzero initial conditions: Singular
  perturbation approximation}}.{\BBCQ}
\newblock
\APACjournalVolNumPages{Appl. Math. Comput.}{353}{}{295--307}.
\PrintBackRefs{\CurrentBib}

\bibitem [\protect \citeauthoryear {%
D’Alessandro%
, Isidori%
\BCBL {}\ \BBA {} Ruberti%
}{%
D’Alessandro%
\ \protect \BOthers {.}}{%
{\protect \APACyear {1974}}%
}]{%
d1974realization}
\APACinsertmetastar {%
d1974realization}%
\begin{APACrefauthors}%
D’Alessandro, P.%
, Isidori, A.%
\BCBL {}\ \BBA {} Ruberti, A.%
\end{APACrefauthors}%
\unskip\
\newblock
\APACrefYearMonthDay{1974}{}{}.
\newblock
{\BBOQ}\APACrefatitle {Realization and structure theory of bilinear dynamical
  systems} {Realization and structure theory of bilinear dynamical
  systems}.{\BBCQ}
\newblock
\APACjournalVolNumPages{SIAM Journal on Control}{12}{3}{517--535}.
\PrintBackRefs{\CurrentBib}

\bibitem [\protect \citeauthoryear {%
Flagg%
\ \BBA {} Gugercin%
}{%
Flagg%
\ \BBA {} Gugercin%
}{%
{\protect \APACyear {2015}}%
}]{%
flagggug}
\APACinsertmetastar {%
flagggug}%
\begin{APACrefauthors}%
Flagg, G.%
\BCBT {}\ \BBA {} Gugercin, S.%
\end{APACrefauthors}%
\unskip\
\newblock
\APACrefYearMonthDay{2015}{}{}.
\newblock
{\BBOQ}\APACrefatitle {{Multipoint Volterra Series Interpolation and
  $\mathcal{H}_2$ Optimal Model Reduction of Bilinear Systems}} {{Multipoint
  Volterra Series Interpolation and $\mathcal{H}_2$ Optimal Model Reduction of
  Bilinear Systems}}.{\BBCQ}
\newblock
\APACjournalVolNumPages{SIAM J. Matrix Anal. Appl}{36}{2}{549--579}.
\PrintBackRefs{\CurrentBib}

\bibitem [\protect \citeauthoryear {%
{Hartmann}%
, {Sch\"afer-Bung}%
\BCBL {}\ \BBA {} {Th\"ons-Zueva}%
}{%
{Hartmann}%
\ \protect \BOthers {.}}{%
{\protect \APACyear {2013}}%
}]{%
hartmann}
\APACinsertmetastar {%
hartmann}%
\begin{APACrefauthors}%
{Hartmann}, C.%
, {Sch\"afer-Bung}, B.%
\BCBL {}\ \BBA {} {Th\"ons-Zueva}, A.%
\end{APACrefauthors}%
\unskip\
\newblock
\APACrefYearMonthDay{2013}{}{}.
\newblock
{\BBOQ}\APACrefatitle {{Balanced averaging of bilinear systems with
  applications to stochastic control.}} {{Balanced averaging of bilinear
  systems with applications to stochastic control.}}{\BBCQ}
\newblock
\APACjournalVolNumPages{{SIAM J. Control Optim.}}{51}{3}{2356--2378}.
\newblock
\begin{APACrefDOI} \doi{10.1137/100796844} \end{APACrefDOI}
\PrintBackRefs{\CurrentBib}

\bibitem [\protect \citeauthoryear {%
Heinkenschloss%
, Reis%
\BCBL {}\ \BBA {} Antoulas%
}{%
Heinkenschloss%
\ \protect \BOthers {.}}{%
{\protect \APACyear {2011}}%
}]{%
HeinReisAn}
\APACinsertmetastar {%
HeinReisAn}%
\begin{APACrefauthors}%
Heinkenschloss, M.%
, Reis, T.%
\BCBL {}\ \BBA {} Antoulas, A\BPBI C.%
\end{APACrefauthors}%
\unskip\
\newblock
\APACrefYearMonthDay{2011}{}{}.
\newblock
{\BBOQ}\APACrefatitle {{Balanced truncation model reduction for systems with
  inhomogeneous initial conditions}} {{Balanced truncation model reduction for
  systems with inhomogeneous initial conditions}}.{\BBCQ}
\newblock
\APACjournalVolNumPages{Automatica}{47}{3}{559--564}.
\PrintBackRefs{\CurrentBib}

\bibitem [\protect \citeauthoryear {%
Huang%
, Jiang%
\BCBL {}\ \BBA {} Qiu%
}{%
Huang%
\ \protect \BOthers {.}}{%
{\protect \APACyear {2021}}%
}]{%
huang2021splitting}
\APACinsertmetastar {%
huang2021splitting}%
\begin{APACrefauthors}%
Huang, Y.%
, Jiang, Y\BHBI L.%
\BCBL {}\ \BBA {} Qiu, Z\BHBI Y.%
\end{APACrefauthors}%
\unskip\
\newblock
\APACrefYearMonthDay{2021}{}{}.
\newblock
{\BBOQ}\APACrefatitle {{Splitting model reduction for bilinear control
  systems}} {{Splitting model reduction for bilinear control systems}}.{\BBCQ}
\newblock
\APACjournalVolNumPages{Asian Journal of Control}{}{}{}.
\PrintBackRefs{\CurrentBib}

\bibitem [\protect \citeauthoryear {%
Lofberg%
}{%
Lofberg%
}{%
{\protect \APACyear {2004}}%
}]{%
lofberg2004yalmip}
\APACinsertmetastar {%
lofberg2004yalmip}%
\begin{APACrefauthors}%
Lofberg, J.%
\end{APACrefauthors}%
\unskip\
\newblock
\APACrefYearMonthDay{2004}{}{}.
\newblock
{\BBOQ}\APACrefatitle {{YALMIP: A toolbox for modeling and optimization in
  MATLAB}} {{YALMIP: A toolbox for modeling and optimization in
  MATLAB}}.{\BBCQ}
\newblock
\BIn{} \APACrefbtitle {2004 IEEE international conference on robotics and
  automation (IEEE Cat. No. 04CH37508)} {2004 ieee international conference on
  robotics and automation (ieee cat. no. 04ch37508)}\ (\BPGS\ 284--289).
\PrintBackRefs{\CurrentBib}

\bibitem [\protect \citeauthoryear {%
Redmann%
}{%
Redmann%
}{%
{\protect \APACyear {2016}}%
}]{%
redmannPhD}
\APACinsertmetastar {%
redmannPhD}%
\begin{APACrefauthors}%
Redmann, M.%
\end{APACrefauthors}%
\unskip\
\newblock
\APACrefYear{2016}.
\unskip\
\newblock
\APACrefbtitle {{Model Order Reduction Techniques Applied to Evolution
  Equations with Lévy Noise}} {{Model Order Reduction Techniques Applied to
  Evolution Equations with Lévy Noise}}\ \APACtypeAddressSchool {\BUPhD}{}{}.
\unskip\
\newblock
\APACaddressSchool {}{{Otto-von-Guericke-Universität Magdeburg}}.
\PrintBackRefs{\CurrentBib}

\bibitem [\protect \citeauthoryear {%
Redmann%
}{%
Redmann%
}{%
{\protect \APACyear {2018}}%
}]{%
redmannspa2}
\APACinsertmetastar {%
redmannspa2}%
\begin{APACrefauthors}%
Redmann, M.%
\end{APACrefauthors}%
\unskip\
\newblock
\APACrefYearMonthDay{2018}{}{}.
\newblock
{\BBOQ}\APACrefatitle {{Type II singular perturbation approximation for linear
  systems with Lévy noise}} {{Type II singular perturbation approximation for
  linear systems with Lévy noise}}.{\BBCQ}
\newblock
\APACjournalVolNumPages{SIAM J. Control Optim.}{56}{3}{2120--2158}.
\PrintBackRefs{\CurrentBib}

\bibitem [\protect \citeauthoryear {%
Redmann%
}{%
Redmann%
}{%
{\protect \APACyear {2020}}%
{\protect \APACexlab {{\protect \BCnt {1}}}}}]{%
redstochbil}
\APACinsertmetastar {%
redstochbil}%
\begin{APACrefauthors}%
Redmann, M.%
\end{APACrefauthors}%
\unskip\
\newblock
\APACrefYearMonthDay{2020{\protect \BCnt {1}}}{}{}.
\newblock
{\BBOQ}\APACrefatitle {{Energy estimates and model order reduction for
  stochastic bilinear systems}} {{Energy estimates and model order reduction
  for stochastic bilinear systems}}.{\BBCQ}
\newblock
\APACjournalVolNumPages{Int. J. Control}{93}{8}{1954--1963}.
\PrintBackRefs{\CurrentBib}

\bibitem [\protect \citeauthoryear {%
Redmann%
}{%
Redmann%
}{%
{\protect \APACyear {2020}}%
{\protect \APACexlab {{\protect \BCnt {2}}}}}]{%
redmannstochbilspa}
\APACinsertmetastar {%
redmannstochbilspa}%
\begin{APACrefauthors}%
Redmann, M.%
\end{APACrefauthors}%
\unskip\
\newblock
\APACrefYearMonthDay{2020{\protect \BCnt {2}}}{}{}.
\newblock
{\BBOQ}\APACrefatitle {{A new type of singular perturbation approximation for
  stochastic bilinear systems}} {{A new type of singular perturbation
  approximation for stochastic bilinear systems}}.{\BBCQ}
\newblock
\APACjournalVolNumPages{Math. Control. Signals, Syst.}{32}{}{129--156}.
\PrintBackRefs{\CurrentBib}

\bibitem [\protect \citeauthoryear {%
Redmann%
}{%
Redmann%
}{%
{\protect \APACyear {2021}}%
}]{%
h2_bil}
\APACinsertmetastar {%
h2_bil}%
\begin{APACrefauthors}%
Redmann, M.%
\end{APACrefauthors}%
\unskip\
\newblock
\APACrefYearMonthDay{2021}{}{}.
\newblock
{\BBOQ}\APACrefatitle {{Bilinear systems -- A new link to $\mathcal H_2$-norms,
  relations to stochastic systems and further properties}} {{Bilinear systems
  -- A new link to $\mathcal H_2$-norms, relations to stochastic systems and
  further properties}}.{\BBCQ}
\newblock
\APACjournalVolNumPages{{SIAM J. Control Optim.}}{59}{4}{2477--2497}.
\PrintBackRefs{\CurrentBib}

\bibitem [\protect \citeauthoryear {%
Redmann%
\ \BBA {} Benner%
}{%
Redmann%
\ \BBA {} Benner%
}{%
{\protect \APACyear {2018}}%
}]{%
redSPA}
\APACinsertmetastar {%
redSPA}%
\begin{APACrefauthors}%
Redmann, M.%
\BCBT {}\ \BBA {} Benner, P.%
\end{APACrefauthors}%
\unskip\
\newblock
\APACrefYearMonthDay{2018}{}{}.
\newblock
{\BBOQ}\APACrefatitle {{Singular Perturbation Approximation for Linear Systems
  with L\'evy Noise}} {{Singular Perturbation Approximation for Linear Systems
  with L\'evy Noise}}.{\BBCQ}
\newblock
\APACjournalVolNumPages{Stochastics and Dynamics}{18}{4}{}.
\PrintBackRefs{\CurrentBib}

\bibitem [\protect \citeauthoryear {%
Redmann%
\ \BBA {} Pontes~Duff%
}{%
Redmann%
\ \BBA {} Pontes~Duff%
}{%
{\protect \APACyear {2022}}%
}]{%
martin_igor}
\APACinsertmetastar {%
martin_igor}%
\begin{APACrefauthors}%
Redmann, M.%
\BCBT {}\ \BBA {} Pontes~Duff, I.%
\end{APACrefauthors}%
\unskip\
\newblock
\APACrefYearMonthDay{2022}{}{}.
\newblock
{\BBOQ}\APACrefatitle {{Full state approximation by Galerkin projection reduced
  order models for stochastic and bilinear systems}} {{Full state approximation
  by Galerkin projection reduced order models for stochastic and bilinear
  systems}}.{\BBCQ}
\newblock
\APACjournalVolNumPages{Appl. Math. Comput.}{420}{}{}.
\PrintBackRefs{\CurrentBib}

\bibitem [\protect \citeauthoryear {%
Schröder%
\ \BBA {} Voigt%
}{%
Schröder%
\ \BBA {} Voigt%
}{%
{\protect \APACyear {2020}}%
}]{%
Matze_inhom}
\APACinsertmetastar {%
Matze_inhom}%
\begin{APACrefauthors}%
Schröder, C.%
\BCBT {}\ \BBA {} Voigt, M.%
\end{APACrefauthors}%
\unskip\
\newblock
\APACrefYearMonthDay{2020}{}{}.
\newblock
{\BBOQ}\APACrefatitle {{Balanced Truncation Model Reduction with A Priori Error
  Bounds for LTI Systems with Nonzero Initial Value}} {{Balanced Truncation
  Model Reduction with A Priori Error Bounds for LTI Systems with Nonzero
  Initial Value}}.{\BBCQ}
\newblock
\APACjournalVolNumPages{arXiv preprint: 2006.02495}{}{}{}.
\PrintBackRefs{\CurrentBib}

\bibitem [\protect \citeauthoryear {%
Zhang%
\ \BBA {} Lam%
}{%
Zhang%
\ \BBA {} Lam%
}{%
{\protect \APACyear {2002}}%
}]{%
morZhaL02}
\APACinsertmetastar {%
morZhaL02}%
\begin{APACrefauthors}%
Zhang, L.%
\BCBT {}\ \BBA {} Lam, J.%
\end{APACrefauthors}%
\unskip\
\newblock
\APACrefYearMonthDay{2002}{}{}.
\newblock
{\BBOQ}\APACrefatitle {On {$H_2$} model reduction of bilinear systems} {On
  {$H_2$} model reduction of bilinear systems}.{\BBCQ}
\newblock
\APACjournalVolNumPages{Automatica}{38}{2}{205--216}.
\PrintBackRefs{\CurrentBib}

\end{thebibliography}

\end{document}